\documentclass[reqno]{article}
\usepackage{authblk}
\usepackage{amssymb,latexsym,amsmath,amsfonts}
\usepackage{graphicx}
\usepackage{color,enumerate}
\usepackage{amsmath,amsfonts,amssymb}
\usepackage[mathscr]{eucal}
\usepackage{bbm}
\usepackage{booktabs}
\usepackage{float}
\usepackage{hyperref}
\usepackage{amsthm}
\usepackage[mathscr]{eucal}
\usepackage{paralist}
\usepackage{rotating}
\usepackage{multirow}
\usepackage[left=1.25in, right=1.25in,bottom=1.25in]{geometry}
\newtheorem{theorem}{Theorem}
\newtheorem{lemma}{Lemma}

\theoremstyle{definition}
\newtheorem{definition}[theorem]{Definition}

\newtheorem{example}[theorem]{Example}
\theoremstyle{remark}

\newcommand {\Cov}{\mbox{Cov}}

\DeclareMathOperator{\1}{\mathbbm 1}
\DeclareMathOperator{\R}{{\mathbb R}}                % reals
\DeclareMathOperator{\N}{{\mathbb N}}                % integer
\renewcommand{\P}{\mathbbm{P}}
\newcommand{\RN}[1]{\uppercase\expandafter{\romannumeral#1}}

\newcommand{\john}[1]{{\color{black}{#1}}}
\definecolor{darkred}{rgb}{0.8, 0, 0}
\newcommand{\Denis}[1]{\textcolor{black}{#1}}

\newcommand{\revision}[1]{\textcolor{black}{#1}}

\newcommand\correspondingauthor{\thanks{Corresponding author.}}
\title{Probabilistic foundations of spatial mean-field models in ecology and applications}
\author[1]{Denis D. Patterson\correspondingauthor}
\affil[1]{Department of Mathematics, Brandeis University (denispatterson@brandeis.edu)}
\author[2]{Simon A. Levin}
\affil[2]{Department of Ecology and Evolutionary Biology, Princeton University (slevin@princeton.edu)}
\author[3]{ A. Carla Staver}
\affil[3]{Department of Ecology and Evolutionary Biology, Yale University (carla.staver@yale.edu)}
\author[1,4]{Jonathan D. Touboul}
\affil[4]{Volen National Center for Complex Systems, Brandeis University (jtouboul@brandeis.edu)}
\date{\today}
\footnotetext{This study was partially supported by NSF DMS 1615585 \& 1951358 (SAL), NSF DMS 1615531 \& 1951394 (ACS) and NSF DMS 1951369 (JDT).}
\begin{document}
	
	\maketitle
	\textbf{Update:} This version of the paper was updated in May 2021 to provide more detail around some estimates used in the proofs of Theorems \ref{thm_convergence_Kstates} and \ref{thm.exist_unique_spatial} (see \eqref{eq.exposition}). The authors would like to thank Philippe Robert for his helpful remarks which aided in improving the exposition of the proofs.
	\begin{abstract}
		Deterministic models of vegetation often summarize, at a macroscopic scale, a multitude of intrinsically random events occurring at a microscopic scale. We bridge the gap between these scales by demonstrating convergence to a mean-field limit for a general class of stochastic models representing each individual ecological event in the limit of large system size. The proof relies on classical stochastic coupling techniques that we generalize to cover spatially extended interactions. The mean-field limit is a spatially extended non-Markovian process characterized by nonlocal integro-differential equations describing the evolution of the probability for a patch of land to be in a given state (the generalized Kolmogorov equations of the process, GKEs). We thus provide an accessible general framework for spatially extending many classical finite-state models from ecology and population dynamics. We demonstrate the practical effectiveness of our approach through a detailed comparison of our limiting spatial model and the finite-size version of a specific savanna-forest model, the so-called Staver-Levin model. There is remarkable dynamic consistency between the GKEs and the finite-size system, in spite of almost sure forest extinction in the finite-size system. To resolve this apparent paradox, we show that the extinction rate drops sharply when nontrivial equilibria emerge in the GKEs, and that the finite-size system's quasi-stationary distribution (stationary distribution conditional on non-extinction) closely matches the bifurcation diagram of the GKEs. Furthermore, the limit process can support periodic oscillations of the probability distribution, thus providing an elementary example of a jump process that does not converge to a stationary distribution. In spatially extended settings, environmental heterogeneity can lead to waves of invasion and front-pinning phenomena.
	\end{abstract}
	\newpage
	\section{Introduction}\label{sec.intro}
	Bridging the microscopic and macroscopic scales in theoretical ecology is an important challenge. Detailed microscopic models describing single organisms are extremely useful for their realistic interpretations and quantitative match to data but their complexity often precludes a detailed mathematical analysis and obscures the key mechanisms at play \cite{andela2017human,fisher2018vegetation}. Simple models on the other hand allow for an in-depth mathematical understanding, but necessitate simplifications and hypotheses that limit their predictive ability \cite{staver_levin_2012}. Despite these limitations, simple systems based on ordinary differential equations have proven extremely useful in ecology.
	
	We propose to bridge the gap between ecological models with macroscopic viewpoints and microscopic descriptions of stochastic transitions. \john{While the methods presented here are quite general, we apply them to }a vegetation ecosystem and the Staver-Levin model of tropical vegetation cover for definiteness~\cite{staver2011tree,staver_levin_2012,touboul2018complex}. The Staver-Levin model \john{is a system of nonlinear differential equations describing the evolution of the fraction of landscape covered by grass, savanna trees or forest trees. The two tree species differ in their birth and death rates, as well as in the way they are affected by fires. Fires, carried by grass, kill (or burn) forest trees but not savanna trees; adult savanna trees resist fires, while saplings, \emph{top-killed} by fires (i.e. burned but with the ability to resprout later), have their maturation into adult trees delayed. Each species also reproduces and dies with specific rates.} This model displays complex dynamical behaviors, including multistability, limit cycles, and homoclinic and heteroclinic orbits. These complex dynamical structures are affected by the  presence of noise, possibly leading to stochastic resonances for Brownian perturbations~\cite{touboul2018complex}. \john{From the ecological viewpoint, the Staver-Levin model summarizes, at a macroscopic scale, a variety of ``microscopic'' events arising at different spatial locations: seed dispersal leading to the birth of a new tree, growth of a new tree, occurrence of fires, death, etc. These events are intrinsically random, localized, and depend on local interactions (e.g., fire spread or seed dispersal). A model explicitly accounting for these stochastic dynamics is essential to better understand how randomness affects the behavior of these systems, and how dynamical behaviors are impacted by spatial interaction; these are two particularly important issues with regard to understanding ecosystem persistence and the potential impacts of climate change on tropical vegetation distributions.} 
	
	There is an extensive literature regarding limits of interacting particle systems, dating back to the early works of Bernoulli~\cite{bernoulli1738hydrodynamica}, Clausius~\cite{clausius1857art} and the celebrated work of Boltzmann~\cite{boltzmann1871warmegleichgewicht} on the kinetic theory of gases. These works, aimed at relating the movement of molecules in a gas to macroscopic quantities such as pressure, have seen remarkable developments in recent years and have been applied to a variety of models (see, e.g., \cite{cabana2018large,cox2013voter,dai1996mckean,durrett1994particle,neuhauser1994long,sznitman1991topics}). The aforementioned mathematical framework has been widely used to derive and analyze spatially extended models in neuroscience \cite{luccon2014mean,robert2016dynamics,touboul2014propagation,touboul2014spatially} and, to a lesser extent, in ecology \cite{durrett2009coexistence}. The recent works of Durrett and Ma \cite{durrett2018heterogeneous,durrett2015coexistence} are important contributions and are most closely related to the present work. The authors consider a two state version of the Staver-Levin model on a toroidal lattice and by taking an appropriately scaled spatial limit, they show convergence in probability of their interacting particle system to the solution of an integro-differential equation (IDE). They obtain a coexistence result by analyzing the resulting IDE and provide bounds on the coexistence time in terms of the system size. We adopt an alternative approach based on stochastic coupling methods \cite{dobrushin1970prescribing,mckean1966class,sznitman1991topics,touboul2014propagation} that allows us \Denis{to} directly demonstrate convergence of the spatial Markov jump process to a McKean-Vlasov jump process (i.e. a process whose transition rates depend on the law of the process itself). Our framework incorporates nonlocal interactions between a finite number of species, heterogeneity from random initial particle placement, and can be applied to many other models in ecology and population dynamics.
	
	\Denis{We suppose that the locations of the vegetation, which we refer to as sites, are randomly distributed on the domain according to some probability measure $q$ and we consider a scaling limit in which sites are successively added to the domain according to this measure. This general model covers two cases relevant for applications: \revision{(i.) \emph{mesoscale models} which resolve the fine spatial structures of the interactions while still taking into account collective effects ($q$ is a continuous distribution), and (ii.) \emph{macroscale models} which lose the fine spatial organization of the interaction between sites in favor of collective stereotyped interactions between a number of isolated vegetation populations ($q$ is a finite combination of Dirac masses). In this way, the macroscale framework is reminiscent of a metapopulation or network model.} 
	}
	
	\john{Within a unified framework \Denis{which encapsulates the Staver-Levin model and its variants}, we show convergence of a class of spatial particle systems to their \emph{mean-field limit},  a non-Markovian stochastic process.} \john{The mean-field limiting process will be characterized by integro-differential equations on the probability density of the process that will be referred to as \emph{generalized Kolmogorov equations} (GKEs). Similar to classical Kolmogorov equations for Markov processes, these equations govern the evolution of the probability distribution of the process. However, contrasting with classical Kolmogorov equations, the GKEs associated with the mean-field limit are nonlinear and, in the mesoscale model, nonlocal}. Our spatial extensions of the Staver-Levin model permit analysis of qualitative properties of practical interest which are not easily accessible via probabilistic methods. For instance, techniques developed to study pattern formation \cite{borgogno2009mathematical,gilad2004ecosystem,gowda2014transitions,lefever1997origin}, wave speeds and invasion phenomena \cite{keitt2001allee,thompson2008plant} and responses to heterogeneous environments \cite{goel2018dispersal,li2019spatial} in spatially extended ecological models are readily applicable to the GKEs. We demonstrate waves of invasion numerically in the \Denis{mesoscale model} (see Figure \ref{fig.maxwell}A) and \Denis{this model} can also naturally incorporate environmental heterogeneity via the choice of the initial site distribution (see Figure \ref{fig.maxwell}B). In fact, in the mesoscale framework, we show that both our particle system and the GKEs of the corresponding mean-field limit can exhibit front pinning, a phenomenon of much practical interest that has previously been found in PDE based ecological models \cite{goel2018dispersal,van2015resilience}.
	
	From an ecological standpoint, further analysis of the models derived in this paper will enable a more thorough understanding of the determinants of the forest-savanna boundary, particularly in the presence of precipitation gradients, resource limitations, and climate change~\cite{li2019spatial,wuyts2018fronts}. Beyond savanna models, the techniques developed here can further be applied to various models. In biology, these techniques could naturally allow addressing stochastic and nonlocal effects arising in other vegetation models (e.g., generic vegetation models~\cite{klausmeier1999regular}), models of infectious diseases~\cite{capasso2008mathematical}, and embryonic development~\cite{quininao2015local} where differentiation of cells depends on non-local signals and gradients. Other perspectives include applications to social models, and variations of voter models~\cite{cox1991nonlinear} with spatial dimension. In the domain of infectious diseases, our framework could be used to study the transmission of an epidemic taking into account the role of spatial positions and nonlocal transmission in a spatially extended SIR-type model~\cite{lloyd1996spatial}. In that regard, it would be of interest to extend our results to individuals moving in space and determine the impact of motility on viral spread, two particularly timely questions with implications for public health policies. 
	
	\Denis{The paper is organized as follows: In Section \ref{sec:SL} we motivate and define a stochastic Staver-Levin model based on microscale events with nonlocal spatial interactions (referred to as the \emph{finite-size stochastic Staver-Levin model} thereafter). We then outline our main mathematical result, Theorem \ref{thm_convergence_Kstates}, in Section \ref{sec:GeneralModel}; this result shows convergence to a well-defined mean-field limiting process for a general class of spatial particle systems and covers the finite-size Staver-Levin model as a special case. In Section \ref{sec.analysis} we compare the qualitative behavior of the finite-size stochastic Staver-Levin model with the GKEs of the corresponding mean-field limits in a series of numerical experiments. We explain long-term transient behaviors via quasi-stationary distributions (QSDs), demonstrate limit cycles in both spatial and nonspatial versions of our models, and show waves of invasion and front pinning. The proof of Theorem \ref{thm_convergence_Kstates} and supporting mathematical preliminaries are deferred to Section \ref{sec.proofs}. Further details of numerical parameters and routines for reproducibility of all figures in the paper are in Appendix \ref{sec.appendix}.}
	\section{Stochastic Models of Tropical Vegetation Dynamics}\label{sec.multiscale}
	\subsection{Microscopic stochastic Staver-Levin model}\label{sec:SL}
	The Staver-Levin model describes the interaction between savanna trees, forest trees, and grass patches. In this model, grass represents an ``open'' patch in which new trees can grow, but also carries fires that limit the expansion of both savanna and forest trees in distinct manners. At the scale of individual trees and patches of grass, the  model relies on a few elementary assumptions (see Fig.~\ref{fig:Model}):
	\renewcommand{\theenumi}{\roman{enumi}}
	\begin{enumerate}[(i.)]
		\item Forest trees grow on ``open'' patches (currently occupied by grass, saplings or savanna trees), at a rate associated with the amount of seeds available on that open patch, therefore related to the density of trees in its vicinity.
		\item Similarly, savanna trees grow on grass occupied patches at a rate that \Denis{depends} on the local density of adult savanna trees.
		\item Grass carries fires, killing forest trees and delaying the maturation of savanna saplings to adults\footnote{Later versions of the Staver-Levin model additionally allow savanna trees and saplings to be somewhat flammable; we omit this extension here for ease of exposition (cf. \cite{touboul2018complex}) but our methods readily allow the incorporation of this extra feature.}.
	\end{enumerate}
	\begin{figure}
		\centering
		\includegraphics[width=\textwidth]{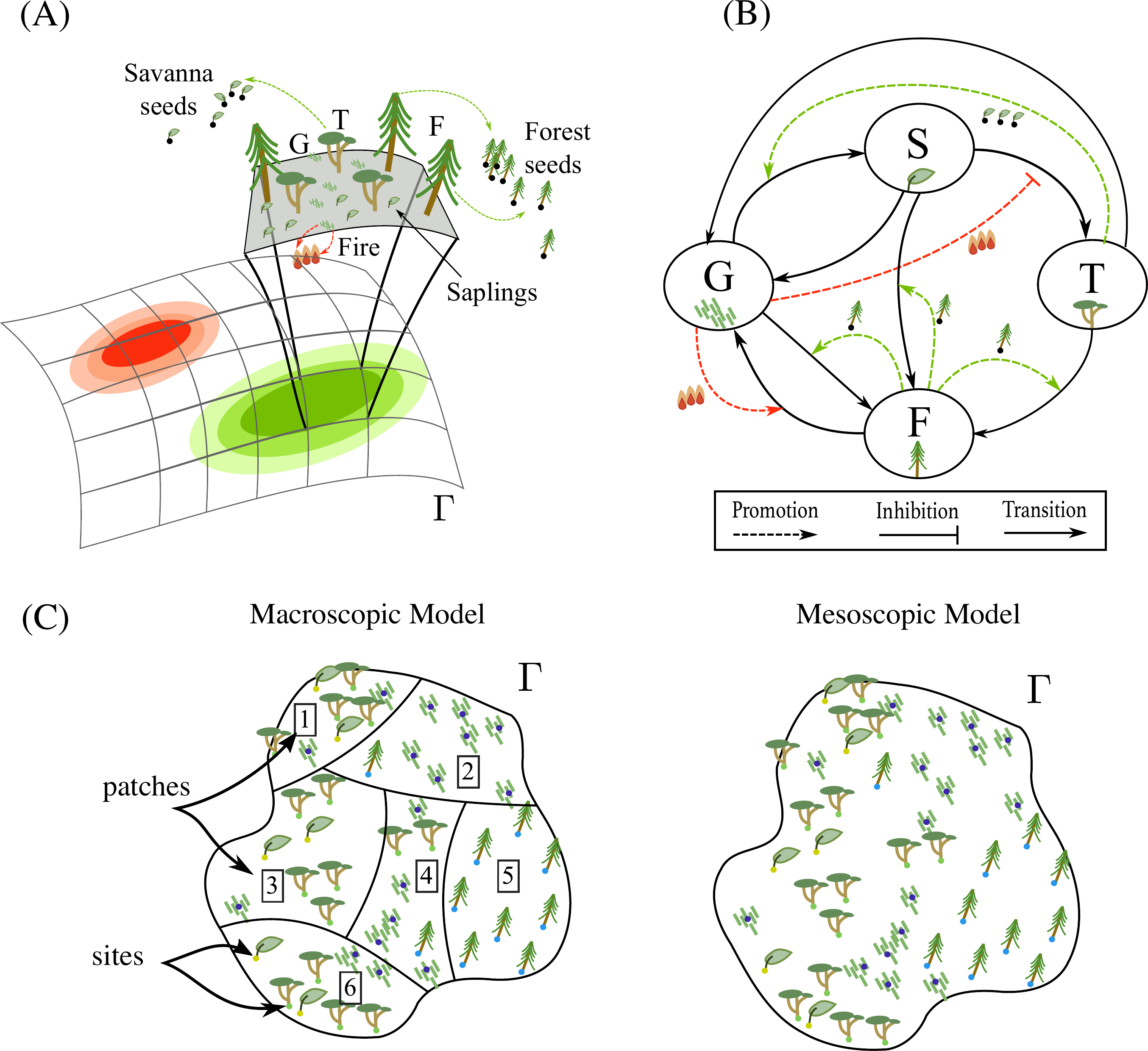}
		\caption{(A) Schematic diagram showing the basic events modeled, on a landscape $\Gamma\subset \R^2$: forest trees $F$ (here, represented as pine trees for the sake of legibility of the diagram, although rainforest trees are not of the coniferous family) and savanna trees $T$ disperse seeds carried by the wind, therefore at a limited reach (green ellipses), growing new forest trees or savanna saplings $S$; land occupied by no tree is by default covered by grass. Fires, essentially carried by grass, ignite depending on the local density of grass (red ellipses), increase the mortality rate of forest trees, and reduce the rate at which savanna saplings mature into adult trees. (B) Simplified interaction diagram between states in the model (see text). (C) Sample model configurations in the macroscopic model with $6$ patches and $8$ sites per patch (left), or for the mesoscopic model with randomly located sites (right).}
		\label{fig:Model}
	\end{figure}
	While the original Staver-Levin model does not explicitly consider spatial effects, each of these interactions depend on the density and locations of patches of grass or trees, and extensive remote sensing data highlights the spatial organization at macroscopic scales of savannas and forest lands~\cite{staver2011tree,staver2019spatial}. 
	
	To account for spatial interactions, we consider a landscape $\Gamma$, a \Denis{Borel set} in $\mathbb{R}^2$. On this landscape, multiple \emph{sites}, thought of as small spatial areas of the typical size of a single tree (\emph{microscopic scale}), allow growth of new trees.  We suppose that $N$ sites are distributed on the landscape $\Gamma$ at locations $(r_i)_{i \in \{1,\dots,N\}}\in \Gamma^N$. The locations are assumed to be independent and identically distributed according to a \Denis{probability measure $q$ on $\Gamma$}. \Denis{A uniform distribution on $\Gamma$ is the most natural choice for the $r_i$'s if we are modeling a homogeneous landscape but certain heterogeneous environmental features may favor more or less vegetative growth in certain parts of the domain (e.g. soil quality/texture) and this may motivate other choices. If $q$ is a continuous probability measure supported on all of $\Gamma$, then we obtain a \emph{mesoscale model} in the limit as $N\to \infty$, the entire domain is populated with sites and interactions depend on the precise location of each site (Fig.~\ref{fig:Model} (C), right). Some applications emphasizing the interaction between multiple regional covers or populations may motivate choosing a discrete measure for $q$ with a finite number of locations\revision{, referred to as patches} in this context. In this case, we retain a discrete spatial structure with multiple interacting populations located at each patch but the number of sites in each patch will still tend to infinity as $N \to \infty$ (see Fig.~\ref{fig:Model} (C), left). \revision{This will be referred to as the \emph{macroscale model}.}}
	
	Each site may switch state at random times with a stochastic intensity that depends on the state of other sites. We denote by $X^i(t)$ the state of site $i$ at time $t$, and label $G$ the grass state, $S$ the savanna sapling state, $T$ the adult savanna tree state and $F$ the forest tree state. The transitions of a given site between states are parameterized as follows (see Fig.~\ref{fig:Model} (A--B)):
	\begin{enumerate}[(i.)]
		\item \Denis{A savanna sapling grows from a site $i$ currently covered by grass with a rate depending on the total amount of savanna tree seeds available at location $r_i$. We thus introduce a savanna seed dispersal kernel, denoted $J_S(r,r')$, corresponding to the rate at which a savanna seed travels from $r'$ to $r$, yielding a transition rate from grass to sapling at site $i$ given by:}
		\[
		\frac{1}{N} \sum_{j = 1}^N J_S(r_i,r_j) \mathbbm{1}_{\{X^j(t) = T\}}
		\]
		where \Denis{$\mathbbm{1}_{\{X^j(t) = T\}}$} is the indicator function, equal to $1$ when \Denis{$X^j(t)=T$} and $0$ otherwise. The transition rate is renormalized by the total number of sites to ensure that it remains within fixed \Denis{bounds} as $N$ is varied --- this scaling can also be considered as a scaling of time. 
		
		\item Similarly, a forest tree grows from a site $i$ currently covered by grass, sapling or savanna tree (owing to the assumption that forest trees are competitively dominant to grass and savanna trees alike) with a rate depending on the total amount of seeds available. The number of forest tree seeds available at location $r_i$ depends on a forest \emph{seed dispersal} kernel $J_F(r,r')$, yielding the transition rate:
		\[
		\frac{1}{N} \sum_{j = 1}^N J_F(r_i,r_j) \mathbbm{1}_{\{X^j(t) = F\}}.
		\]
		\item Forest trees die and are replaced by grassy patches either due to fires, fueled by grass, \Denis{or background non-fire mortality}. The rate at which site $i$ of forest type switches to grass thus depends on the flammable cover available in the vicinity of site $i$ and their capacity to transmit fires; this transition rate is given by:
		\[
		\phi\left(\frac{1}{N} \sum_{j = 1}^N W(r_i,r_j) \mathbbm{1}_{\{X^j(t) = G\}}\right).
		\]
		\Denis{The \emph{forest mortality function} $\phi$ is a smooth, sharply increasing sigmoidal function of the flammable cover affecting a site at location $r_i$. The map $\phi$ incorporates baseline mortality of forest trees $\phi(0)>0$. More detailed models of fire propagation in savannas have modeled fire \revision{propagation using percolation,} and previous studies have shown that fire spread in these models has a sharp sigmoidal profile as a function of the flammable cover~\cite{schertzer2015implications}. Our approach regarding fire propagation thus remains phenomenological and aims to capture the qualitative character of fire spread while retaining analytic tractability.}
		
		\item Savanna saplings mature into adult savanna trees. The associated rate of transition is affected by fires that, by top-killing saplings, delay their maturation, and the probability of being affected by a fire depends, again, on the local flammable cover and their capacity to transmit fires. The maturation rate of a sapling into an adult savanna tree is thus given by:
		\[
		\omega\left( \frac{1}{N} \sum_{j = 1}^N W(r_i,r_j) \mathbbm{1}_{\{X^j(t) = G\}} \right).
		\]
		where $\omega$ is a sigmoid function (smooth, decreasing and bounded) quantifying the delayed maturation associated with top-killed saplings. \Denis{Once more, the sigmoidal form of $\omega$ aims to capture the qualitative properties of fire spread without explicitly modeling the underlying fire process.}
		\item Savanna saplings and trees die at constant rates, independent of the state of the system, denoted by $\mu$ and $\nu$ respectively. 
	\end{enumerate}
	\Denis{The transitions for the \emph{finite-size stochastic Staver-Levin model} outlined above are summarized in Table~\ref{tab:table1} (along with the corresponding rates for the mean-field limiting process). }
	\begin{table}
		\begin{center}
			\begin{tabular}{@{}cccc@{}}
				\textbf{Transition}              &       & \textbf{Rate in finite-size system} & \textbf{Rate in mean-field limit} \\ \toprule
				$G \to S$ &      &   $\displaystyle{\frac 1 N \sum_{j = 1}^N J_S(r_i,r_j)\mathbbm{1}_{\{X^j(t) = T\}}}$    & $\int_{\Gamma}J_S(r_i,r')P_T(t,r')\,dq(r')$     \\ \midrule
				$G, S, T \to F$ &    &    $\displaystyle{\frac 1 N\sum_{j = 1}^N J_F(r_i,r_j)\mathbbm{1}_{\{X^j(t) = F\}}}$  & $\int_{\Gamma}J_F(r_i,r')P_F(t,r')\,dq(r')$     \\ \midrule
				$F \to G$ &      &     $\displaystyle{\phi\left(\frac{1}{N} \sum_{j = 1}^N W(r_i,r_j) \mathbbm{1}_{\{X^j (t) = G\}}\right)}$   & $\displaystyle{\phi\left( \int_{\Gamma} W(r_i,r') P_G(t,r')\,dq(r') \right)}$    \\ \midrule
				$S \to T$ &      &     $\displaystyle{\omega\left( \frac 1 N \sum_{j=1}^N W(r_i,r_j) \mathbbm{1}_{\{X^j (t) = G\}} \right)}$   & $\displaystyle{\omega\left(\int_{\Gamma} W(r_i,r') P_G(t,r')\,dq(r')\right)}$     \\ \midrule
				$S \to G$ &     &      $\mu$  & $\mu$  \\ \midrule
				$T \to G$ &     &     $\nu$    & $\nu$    \\ \bottomrule
			\end{tabular}
			\vspace*{5pt}
			\caption{Transition rates in the finite-size stochastic Staver-Levin model (column 2) and the corresponding transition rates of the mean-field limiting process (column 3). $P_x(t,r)=\mathbb{P}[\bar{X}(t,r) = x]$ denotes the probability that the mean-field process $\bar{X}$ is in state $x \in \{G,S,T,F\}$ at time $t$ and location $r$.} 
			\label{tab:table1}
		\end{center}
	\end{table}
	\Denis{The study of the finite-size model defined above is particularly simple. Indeed, these systems are continuous-time finite-state Markov processes (with $4^N$ possible states for the system) with bounded transition rates. Due to the large dimension of the process, the Kolmogorov equations are impractical to write and solve, but elementary considerations reveal that the system will always tend to the all-grass state. The state where all sites are grass is an absorbing state, since the transition rates from that state to saplings or trees are equal to $0$ when no forest or savanna trees are present in the system. Moreover, it is easy to see that each state can reach the all-grass absorbing state in a finite number of steps, with a rate bounded away from $0$. A classical result of Markov processes thus ensures that, with probability 1, the Markov chain is absorbed in finite time by the all-grass state. Therefore, the stationary dynamics of the finite-size system will always be trivial~\cite{darroch1967quasi}. This trivial stationary solution sharply contrasts with the complex dynamics of the Staver-Levin model~\cite{touboul2018complex}. We will thus pay particular attention to the transient dynamics and their duration in Section~\ref{sec.multi_comparison}, two elements of particular interest in ecological modeling that have garnered increasing research interest in recent years~\cite{hastings2018transient}.}
	
	\subsection{A General Convergence Result}\label{sec:GeneralModel}
	\john{The finite-size stochastic Staver-Levin model is one exemplar of a wide class of models arising in ecology and physics in which a large number $N$ of sites (agents, particles, etc.) with fixed spatial locations interact according to the state of their neighbors. Mathematically, we will study a general class of Markovian models where each of the $N$ sites $i\in\{1,\dots,N\}$ can be in one of $K$ states. Each site $i\in\{1,\dots,N\}$ is located at an i.i.d. spatial position $r_i$ drawn from a probability measure $q$ on $\Gamma \in \mathcal{B}(\R^2)$ (where $\mathcal{B}(\cdot)$ denotes the Borel $\sigma$-algebra on a given set). The state of each site at time $t\in\R_+$ is denoted $X^i(t)$, and the set of possible states is denoted $S^K$ and has $K$ elements (in the Staver-Levin model, these are the vegetation types $\{G,S,T,F\}$). The initial state of \revision{each site is independent of that of all other sites and each sites initial state has the same space-dependent law satisfying the following regularity condition:}
		\begin{description}
			\item[$H_{IC}.$] The initial distribution for site $i$ is given by $\xi^i(r_i)$ where the collection of processes $(\xi^j(r))_{j\in \N}$ are independent random variables indexed by space, with probabilities measurable with respect to space. Rigorously, we assume that their distribution is given by a Markov kernel $\mu_0: \Gamma \to \Omega$ with $(\Omega,\mathcal{F},\P)$ the current probability space\footnote{A Markov kernel (see, e.g.~\cite[Chap. 3, Def. 1.1]{revuz2013continuous}) with source $(X,\mathcal{A})$ and target $(\Omega,\mathcal{F})$, denoted $\kappa:X\to \Omega$, is a map $\kappa:X \times \Omega \mapsto [0,1]$ such that:
				\renewcommand{\theenumi}{\roman{enumi}}
				\begin{enumerate}[(i.)]
					\item for every $F\in\mathcal{F}$, the map $x\mapsto \kappa(x,F)$ is $\mathcal{A}$-measurable
					\item for every $x\in A$, the map $\kappa(x,\cdot)$ is a probability measure.
				\end{enumerate} 
				In other words, $\kappa$ associates to every point $x$ a probability measure such that for every measurable set of the probability space, the probability $x\mapsto \kappa(B,x)$ is measurable.}.
		\end{description}
		
		\revision{Following a classical convention, we consider càdlàg versions of our jump processes (i.e., right-continuous with left limits everywhere), and we denote by $X^i(t^-)$ the left-limit of the state of site $i$ at time $t$. The transition rate for site $i$ at location $r_i$ and between any two states $x$ and $y\neq x$ depends on the local densities of some other species at a given time $t$. This rate is thus time dependent, in that it varies with the state of the system. Between the last transition before time $t$ and time $t$, the transition rate is therefore constant (since no site has changed state) and is assumed to be given by nonlinear functions of a weighted sum of other sites' states:}
		\begin{equation}\label{eq:RatesGeneral}
		R_{x,y}^{i,N}\left(X(t^-)\right) = \Phi_{x,y}\left( \frac{1}{N}\sum_{j = 1}^N W_{x,y} \left(r_i,\,r_j \right)\1_{ \{X^j(t^-) = \psi(x,y)\} } \right),
		\end{equation}
		where the kernels $\left\{W_{x,y} :\,(x,y) \in S^K \times S^K \right\}$ weight the nonlocal influence of other sites on the transition rate from state $x$ to state $y$ at a given site. Here, $\Phi_{x,y}:\R_+\mapsto \R_+$ are smooth (nonlinear) functions and the maps $\psi(x,y): S^K \times S^K \mapsto S^K \cup \{\emptyset\}$ are the states the transition $x\to y$ depends on.  For definiteness, state $\emptyset$ is a state not belonging to $S^K$, chosen if the transition rate is independent of any other state, e.g. when the transition cannot occur. For simplicity, we assume that transitions only depend on one other species, but it would be straightforward to generalize to transitions depending on multiple species by making the functions $\Phi_{x,y}$ multivariate. In the 4-species Staver-Levin model, the functions $\Phi_{x,y}$ and $\psi$ are given by:
		\begin{table}[H]
			\centering
			\begin{tabular}{l|c|c|c|c}
				$\Phi_{x,y}(u)$ & $x=G$ & $ S$ & $ T$ & $ F$	 \\ \hline
				$y =G$ & $0$ & $\nu$ & $\mu$ & $\phi(u)$	 \\ \hline
				$y =S$ & $u$ & $0$ & $0$ &	$0$ \\ \hline
				$y =T$ & $0$ & $\omega(u)$ & $0$ & $0$ \\ \hline
				$y =F$ & $u$ & $u$ & $u$ & $0$
			\end{tabular}\qquad 
			\begin{tabular}{l|c|c|c|c}
				$\psi(x,y)$ & $x = G$ & $S$ & $ T$ & $F$	 \\ \hline
				$y = G$ & $\emptyset$ & $\emptyset$ & $\emptyset$ & $G$	 \\ \hline
				$y = S$ & T & $\emptyset$ & $\emptyset$ &	$\emptyset$ \\ \hline
				$y = T$ & $\emptyset$ & $G$ & $\emptyset$ &	$\emptyset$ \\ \hline
				$y = F$ & $F$ & $F$ & $F$ & $\emptyset$
			\end{tabular}
		\end{table}
		\noindent We assume further that:
		\renewcommand{\theenumi}{H\arabic{enumi}}
		\begin{enumerate}
			\item\label{eq.Lipschitz} $\Phi_{x,y} : \mathbb{R}_+ \mapsto \mathbb{R}_+$ is Lipschitz continuous for each $(x,y)$.
			\item\label{eq.bounded_measurable} $W_{x,y} : \Gamma \times \Gamma \mapsto \mathbb{R}_+$ is bounded and Borel-measurable for each $(x,y)$.
		\end{enumerate}
		\begin{theorem}\label{thm_convergence_Kstates} 
			Consider the Markov process with rates given by~\eqref{eq:RatesGeneral} with initial conditions satisfying assumption $H_{IC}$ and rates given by equation \eqref{eq:RatesGeneral} satisfying assumptions H1 and H2.
			
			\textbf{I. Convergence.} For any time $\tau>0$ and any site $i\in\mathbb{N}$, the process $X^{i} = \{X^{i}(t):  {t\in [0,\tau]}\}$ converges in law, as $N\to \infty$, to the process $\bar{X}(r_i) = \{\bar{X}(t,r_i): {t\in [0,\tau]}\}$ where $\bar{X}(t,r)$ is the unique strong solution to the \revision{spatially extended} McKean-Vlasov jump process with transitions occurring at independent, exponentially distributed times with rates given by:
			\begin{equation}\label{eq:MeanFieldMeso_Kstates}
			x \to y \qquad \Phi_{x,y}\left( \int_\Gamma W_{x,\,y} \left(r,\,r'\right)\mathbb{P}\left[\bar{X}(t^-,r') = \psi(x,\,y) \right]dq(r') \right),
			\end{equation}	
			for each pair of states $(x,y) \in S^K \times S^K$ and initial condition with law $\mu_0$.
			
			\textbf{II. Characterization of the limit.} Define $P_{x}(t,r) = \mathbb{P}[ \bar{X}(t,r) = x]$ for each $x \in S^K$. The probability distribution of the process $\bar{X}$ satisfies the system of generalized Kolmogorov equations (GKEs):
			\begin{multline}\label{eq:IDE_Kstates}
			\partial_t {P}_{x} (r) = \sum_{y\neq x} \Phi_{y,x}\left( \int_\Gamma W_{y,\,x} \left(r,\,r'\right)P_{\psi(y,\,x)}(r')\,dq(r') \right)  P_{y}(r) \\
			-  \sum_{y\neq x} \Phi_{x,y}\left( \int_\Gamma W_{x,\,y} \left(r,\,r'\right)P_{\psi(x,\,y)}(r')\,dq(r') \right) P_{x}(r), \quad x \in S^K,
			\end{multline}
			where $P_{y}(r)$ is shorthand for $P_{y}(t^-,r)$ for each $y\in S^K$ and each $(t,r) \in \mathbb{R}_+ \times \Gamma$.
			
			\textbf{III. Propagation of Chaos.} Moreover, any fixed finite subset of sites \\ $(X^{i_1},\dots,X^{i_p})$ converge towards independent variables, i.e. for any collection of $p$ states  $(x_1,\cdots,x_p)$,
			\[ \lim_{N\to\infty}\P[X^{i_1}(t)=x_1,\dots, X^{i_p}(t)=x_p] = \P[\bar{X}(t,r_{i_1})=x_1] \times \cdots \times \P[\bar{X}(t,r_{i_p})=x_p],\]
			for each $t \in [0,\tau]$.
		\end{theorem} 
		
		Theorem \ref{thm_convergence_Kstates} demonstrates that the Markov process given by \eqref{eq:RatesGeneral} is well-approximated in the large $N$ limit by the stochastic jump process $\bar{X}$, which we refer to as the mean-field limit. $\bar{X}$ is not a Markovian jump process: the rate of transition depends explicitly on the law of the process $\bar{X}$, and not only on its current state. These types of processes are referred to as McKean-Vlasov processes and it is generally hard to characterize properties of their solutions~\cite{villani2002review}. However, in this case, Theorem \ref{thm_convergence_Kstates} provides a simple characterization of the solution in the form of a system of 
		\begin{itemize}
			\item integro-differential equations (IDEs) when $q$ is a continuous distribution \\(mesoscopic model), or 
			\item nonlinear ordinary differential equations when $q$ is a finite combination of Dirac masses (macroscale model).
		\end{itemize} 
		We refer to the system of equations given by \eqref{eq:IDE_Kstates} as the  \emph{generalized Kolmogorov equations} (GKEs) for $\bar{X}$ because, similar to classical Kolmogorov equations for Markovian processes, they describe the time evolution of the probability density of the continuous-time finite-state jump process $\bar{X}$. However, they are not standard Kolmogorov equations, since $\bar{X}$ is non-Markovian -- this is due to the dependence of $\bar{X}$'s transition rates on the probability density of the process itself. This dependence makes the GKEs nonlinear, contrasting with classical Kolmogorov equations. The GKEs are also more analytically tractable than the finite-size Markov process \eqref{eq:RatesGeneral} or even than the linear Kolmogorov equations associated to the finite-size system, since these equations are very high dimensional (associated with a $K^N$ system of linear differential equations). Furthermore, the GKEs naturally satisfy a conservation law that was included in the original Staver-Levin model~\cite{staver_levin_2012}, i.e.
		\begin{equation}\label{prob_constraint}
		\sum_{x\in S^K} {P}_x(t,r)=1, \quad\mbox{for each }r \in \Gamma \quad\text{and } t\in [0,\tau].
		\end{equation}
		
		The proof of convergence in Theorem \ref{thm_convergence_Kstates} uses a coupling method consisting of two steps. First, we prove strong existence and uniqueness of solutions to the associated mean-field equation given by \eqref{eq:MeanFieldMeso_Kstates}. Second, we construct a particular solution $\bar{X}^i$ of the limiting equation having the same initial condition and coupling the jump times with those of the finite system, to which the process  $\{X^{i}(t),\, t\in [0,\tau] \}$ converges almost surely. The rate of convergence is also quantified through this coupling argument, and it is shown that the distance between the finite system and their limit decays as $1/\sqrt{N}$. 
		
		The coupling methods we employ are classical and have traditionally been applied to high dimensional Markov processes, often with multiple classes, but typically without spatial structure~\cite{aghajani2018large,dawson1991law,dawson2005balancing,feng1992solutions,graham2009interacting}. It is only recently that we  extended these methods to spatial systems with uncountable state-spaces (i.e. interacting particle systems driven by independent Brownian motions)~\cite{touboul2014propagation,touboul2014spatially}. Some recent works have addressed the question of mean-field convergence with spatial structure but only for nearest neighbor interactions~\cite{thai2015birth}. The Staver-Levin model specifically motivates the extension of these methods to finite-state continuous-time Markov processes with nonlocal interactions, due to the nonlocal and long-range nature of seed dispersal~\cite{clark1999seed,thompson2008plant}. 
		
		A direct application of Theorem~\ref{thm_convergence_Kstates} shows that the stochastic Staver-Levin model with rates given by Table \ref{tab:table1} converges towards a non-Markovian spatially extended McKean-Vlasov jump process with transition rates given by:
		\begin{equation}\label{eq:MeanFieldMeso}
		\begin{cases}
		G \to S &\quad \int_{\Gamma} J_S(r,r')P_T(t,r')\,dq(r')\\
		G, S, T \to F &\quad \int_{\Gamma}J_F(r,r')P_F(t,r')\,dq(r')\\
		F \to G &\quad \phi\left( \int_{\Gamma} W(r,r') P_G(t,r')\,dq(r') \right) \\
		S \to T &\quad \omega\left(\int_{\Gamma} W(r,r') P_G(t,r')\,dq(r') \right)\\
		S \to G &\quad \mu\\
		T \to G &\quad \nu
		\end{cases}
		\end{equation}
		In this case, the GKEs of the mean-field limiting process with transition rates given by \eqref{eq:MeanFieldMeso} are:
		\begin{equation}\label{eq:IDE}
		\begin{cases}
		\partial_t {P}_G (r) &= \mu {P}_S(r)+\nu {P}_T(r) +\phi\left( \int_{\Gamma} W(r,r') {P}_G (r')\,dq(r')\,\right){P}_F(r) \\
		&\qquad - {P}_G(r) \int_{\Gamma} J_S(r,r'){P}_T (r')\,dq(r')-{P}_G(r) \int_\Gamma J_F(r,r'){P}_F(r')\,dq(r')\\ \\
		\partial_t {P}_S(r) &= {P}_G(r) \int_{\Gamma} J_S(r,r'){P}_T(r')\,dq(r')-{P}_S(r)\int_{\Gamma} J_F(r,r'){P}_F(r')\,dq(r')\\
		&\qquad - {P}_S(r)\,\omega\left( \int_{\Gamma} W(r,r') {P}_G(r')\,dq(r') \right) -\mu {P}_S(r) \\ \\
		\partial_t {P}_T(r) &= -\nu {P}_S(r) -{P}_T(r)\int_{\Gamma} J_F(r,r'){P}_F(r')\,dq(r')\\ &\qquad + {P}_S(r)\,\omega\left( \int_{\Gamma} W(r,r') {P}_G(r')\,dq(r')\right)\\ \\
		\partial_t {P}_F(r) &= (1-{P}_F(r))\int_\Gamma J_F(r,r'){P}_F(r')\,dq(r')\\ &\qquad - \phi\left( \int_\Gamma W(r,r') {P}_G(r')\,dq(r')\right){P}_F(r),\\
		\end{cases}
		\end{equation}
		
		Theorem \ref{thm_convergence_Kstates} thus provides a natural candidate for a deterministic spatially extended Staver-Levin model, distinct from classical extensions generally done for ecological models and relying on diffusion operators~\cite{murray2004mathematical}. Our framework also allows consideration of different situations than those covered by Durrett and collaborators~\cite{durrett2009coexistence,durrett2018heterogeneous,durrett1994particle}, in particular regarding arbitrary site distributions and arbitrary smooth kernels and spatial dependence, although our approach is limited to smooth dependence on the states (Lipschitz assumptions H1 and H2). Furthermore, owing to their relationship with a fine microscopic model, the integro-differential equations \eqref{eq:IDE} are ecologically relevant: the spatial integral terms can be readily interpreted as nonlocal effects of seed dispersal and fire propagation. 
		\begin{example}[Macroscale Vegetation Model and the Staver-Levin model as a metapopulation or network model]\label{eg.macroscale}
			Macroscale models correspond to cases where the landscape is described as composed of a finite number $M\in\mathbb{N}$ of patches $j \in \{1,\dots ,M\}$. This case is treated considering that each patch is labeled through a location variable in the space $\mathcal{A}_M := \{1,\dots,M\}$. Now suppose that $q=(q_{j})_{j=1,\dots, M}$ is a probability measure on $\mathcal{A}_M$, and denote by ${P}_x^{j}$ the probability that the population in patch $j$ is in state $x\in\{G,S,T,F\}$. By Theorem \ref{thm_convergence_Kstates}, these probabilities satisfy the following GKEs in the limit as $N\to \infty$:
			\begin{equation}\label{eq.IDE_macroscale}
			\begin{cases}
			\frac{d {P}_G^{j}}{dt}  &= \mu {P}_S^{j}+\nu {P}_T^{j} +\phi\left( \tfrac{1}{M}\sum_{k=1}^M q_{k} W(j,k) {P}_G^{k} \right){P}_F^{j} \\
			&\quad - \frac{{P}_G^j}{M} \sum_{k=1}^M q_{k} J_S(j,k) {P}_T^{k} - \frac{{P}_G^{j}}{M} \sum_{k=1}^M q_k J_F(j,k){P}_F^k\\ \\
			
			\frac{d {P}_S^{j}}{dt} &= \frac{{P}_G^j}{M} \sum_{k=1}^M q_k J_S(j,k) {P}_T^k - \frac{{P}_S(r)}{M}\sum_{k = 1}^M q_k J_F(j,k){P}_F^k\\
			&\quad - {P}_S^j\,\omega\left( \tfrac{1}{M}\sum_{k=1}^M q_k W(j,k) {P}_G^k \right) -\mu {P}_S^j \\ \\
			
			\frac{d {P}_T^{j}}{dt} &= -\nu {P}_S^j - \frac{{P}_T^j}{M}\sum_{k=1}^M q_k J_F(j,k){P}_F^k + {P}_S^j\,\omega\left( \tfrac{1}{M}\sum_{k=1}^M q_k W(j,k) {P}_G^k\right)\\ \\
			\frac{d {P}_F^{j}}{dt} &= \frac{1}{M}(1-{P}_F^j)\sum_{j=1}^M q_k J_F(j,k){P}_F^k - \phi\left( \tfrac{1}{M}\sum_{k=1}^M q_k W(j,k) {P}_G^k\right){P}_F^j,
			\end{cases}
			\end{equation}
			for each $j \in \{1,\dots,M\}$. If $M = 1$ and
			\[
			W(1,1) = 1, \quad  J_S(1,1) = \beta, \quad J_F(1,1) = \alpha,
			\]
			the GKEs are exactly the classical Staver-Levin ODE model~\cite{staver_levin_2012}:
			\begin{equation}\label{eq.Staver_levin_ode}
			\begin{cases}
			\dot{G} &= \mu\,S + \nu\,T + \phi\left( G \right)F - \beta\, G \, T - \alpha\,
			G \,F, \\
			\dot{S} &= \beta\, G\, T - \alpha\,G\,F - S \,\omega\left( G\right) - \mu \,S,\\
			\dot{T} &= - \nu \,S - \alpha\,T\,F + \omega(G)\,S, \\
			\dot{F} &= \alpha\,(1-F)\,F - \phi(G)F,
			\end{cases}
			\end{equation}
			where $G$, $S$, $T$ and $F$ denote the proportions of landscape covered by grass, saplings, savanna trees or forest trees in the original Staver-Levin model, while in our context they represent the probability for a given site to be in one of these states. In other words, the condition imposed in the original Staver-Levin model ensuring that the variables are proportions, i.e.
			\begin{equation}\label{eq.SL_ODE_norm}
			G(t) + S(t) + T(t) + F(t) = 1, \quad \mbox{for all}\quad t \geq 0,
			\end{equation}
			corresponds exactly to the probabilistic constraint \eqref{prob_constraint} obeyed by solutions to the GKEs. Because of the propagation of chaos property (Theorem~\ref{thm_convergence_Kstates} III.), finite subsets of sites become independent in the large $N$ limit and therefore the empirical proportion of landscape covered by a given species on this subset is well approximated by the probability for any given site to be in that state. 
		\end{example}
		\begin{example}[Homogeneous Isotropic Grass-forest model]
			We consider an example of competition between grass and forest (without savanna trees) on a landscape $\Gamma$, a Borel set in $\mathbb{R}^2$. We assume that the landscape is homogeneous (i.e., $q$ is the uniform measure on $\Gamma$). Furthermore, as is typical in applications, we assume $W$ and $J_F$ are radial, i.e. $f(r,r') = \hat{f}(|r-r'|)$ for $f=J_F$ or $W$, and where $|\cdot|$ denotes the Euclidean norm on $\Gamma$. Theorem~\ref{thm_convergence_Kstates} implies that the GKEs associated with the mean-field limit in this case are:
			\begin{equation}\label{eq.IDE_microscale}
			\begin{cases}
			\partial_t {P}_G (r) &= \phi\left( (\hat{W}*P_G) (r)\right){P}_F(r) - {P}_G(r) (\hat{J}_F*P_F)(r)\\
			\partial_t {P}_F(r) &= (1-{P}_F(r))(\hat{J}_F * {P}_F)(r) - \phi\left( (\hat{W} * {P}_G)(r)\right){P}_F(r),\\
			\end{cases}
			\end{equation}
			for each $r \in \Gamma$, with $\left(f * g\right) (r)= \int_{\Gamma} f(\vert r-r'\vert)g(r')\frac{dr'}{Leb(\Gamma)}$ with $Leb(\Gamma)$ denoting the Lebesgue measure of $\Gamma$. The system of equations \eqref{eq.IDE_microscale} is also subject to the constraint:
			\begin{equation*}
			{P}_G(t,r) + P_F(t,r) = 1, \quad\mbox{for each }r \in \Gamma \quad\text{and } t\in [0,\tau]. 
			\end{equation*}
	\end{example}}
	\section{Spatially Extended Staver-Levin Model: Analysis of Solutions}\label{sec.analysis}
	\john{Our finite-size Staver-Levin models are Markov processes with an absorbing (or trapping) state that can be reached from any other state. Therefore, the unique stationary solution for these processes is the all-grass state, and, with probability 1, this state is reached in finite time. However, this absorption can arise after long transients; this transient behavior is particularly important in applications like ecology where time scales are typically long~\cite{hastings2018transient}. In particular, this finite-time absorption is in sharp contrast with the remarkable complexity of the original Staver-Levin model, where the all-grass state may be unstable and other equilibria or periodic orbits emerge \cite{touboul2018complex}. Thus it is of evident interest to study the consistency between finite-size tree extinction and non-extinction in the mean-field limit for both the macroscale and mesoscale frameworks. All details related to numerics and coding for this section can be found in Appendix \ref{sec.appendix}.}
	\subsection{Macroscopic Markov Model and  Generalized Staver-Levin Model}\label{sec.multi_comparison}
	We first investigate absorption properties and the consistency between the quasi-stationary distributions (behavior of the finite-size system prior to extinction) and the GKEs of the mean-field limiting process. We then demonstrate long run transient periodic solutions to the particle system in a parameter regime for which the limiting process has a periodic law. 
	\subsubsection{Mean-field Behavior in a Single Patch}
	We analyze in detail the absorption property of the macroscale model \Denis{introduced in Example \ref{eg.macroscale} with a single patch in the forest-grass subsystem. In the notation of Example \ref{eg.macroscale}, suppose}
	\[
	W(1,1) = 1, \quad J_F(1,1) = \bar{J}>0.
	\]
	\Denis{Since we restrict our attention to the forest-grass subsystem, the finite-size} system $X$ is a classical two-state Markov process in dimension $N$ with (for each site $i\in \{1,\dots,N\}$) independent, exponentially distributed transitions with rates given by:
	\begin{equation}\label{eq:TwoStatesOnePatchNet}
	\begin{cases}
	G \to F  \qquad \bar{J}\, \hat{P}^N_F(t) &\text{(forest tree seed dispersal)}\\
	F \to G \qquad \phi\left(1-\hat{P}^N_F(t)\right) &\text{(burning of forest trees)}, \\
	\end{cases}
	\end{equation}
	where $\hat{P}^N_F=N^{-1} \sum_{j=1}^N \mathbbm{1}_{\{X^{j}(t)=F\}}$ is the fraction of forest sites (or empirical probability of the forest state across all sites) and $\bar{J}$ is the birth rate of forest trees. The mean-field limit $\bar{X}$ is a two-state non-Markovian process with independent exponentially distributed transitions that occur with rates:
	\begin{equation}\label{eq:TwoStatesOnePatchMF}
	\begin{cases}
	G \to F  \qquad &\bar{J} P_F(t) \\
	F \to G  &\phi\left( 1-P_F(t) \right), \\
	\end{cases}
	\end{equation}
	where $P_F(t)$ and $P_G(t)$ are the probabilities that $\bar{X}$ is in the forest state or grass state respectively. These probabilities can be computed by solving the GKEs of the mean-field limit:
	\begin{align}\label{eq.meanfield_diff}
	\frac{d}{dt}P_G(t) &= \left(1-P_G(t)\right) \phi\left( P_G(t) \right) - \bar{J}\, P_G(t) \left(1 - P_G(t)\right)
	\end{align}
	with $P_G(t)=1-P_F(t)$. \Denis{If $\phi(0)>0$ and $\bar{J}>0$}, the all grass state is the unique stationary distribution for the finite-size particle system and, from any initial distribution, absorption occurs in finite time with probability 1. \john{Consistent with this, $P_G=1$ is always a fixed point of the GKEs for the corresponding mean-field limit. Moreover, this fixed point is attractive for forest-tree birth rate small enough, but loses stability when the tree birth rate exceeds tree death rate in the all grass state (i.e. $\bar{J}<\phi(1)$). Therefore, for larger birthrates, the mean-field limit will not converge towards the all-grass state, the unique attractor of the finite system.}
	
	The bifurcation diagram for the GKEs of the mean-field limit process with rates given by \eqref{eq:TwoStatesOnePatchMF} is computed as a function of the forest tree birth-rate $\bar{J}$ (Figure \ref{two_state_expected_bifurcations}, red); it features two saddle-nodes and a transcritical bifurcation that split the parameter space into four main regimes: 
	\begin{itemize}
		\item for forest birth rate small enough, the only stable equilibrium is the all-grass state
		\item A saddle-node bifurcation (arising at $\bar{J} \approx 0.55$ for our choice of parameters) leads to the emergence of a stable forest-dominated equilibrium as well as an unstable fixed point. The all-grass equilibrium conserves its stability but, depending on the initial grass cover relative to the unstable equilibrium, the system either converges to the all-grass solution or to the stable forest-dominated equilibrium.
		\item As the forest birthrate is further increased, the all-grass state loses stability (at $\bar J=\phi(1)$) in favor of a grass-dominated equilibrium through a transcritical bifurcation. In that case, whatever the initial condition, the limit system never reaches the all-grass equilibrium. Moreover, as $\bar{J}$ is increased, the grass cover of the grass-dominated equilibrium decreases progressively as the forest tree birthrate increases, before disappearing through a second saddle-node bifurcation.
		\item Beyond this second saddle-node bifurcation, the only stable equilibrium persisting is the  forest-dominated equilibrium. 
	\end{itemize}
	
	The blue stars in Figure \ref{two_state_expected_bifurcations} show the state of the finite-size system at time $t = 100$ from simulations of the one-patch 2-state Markov chain with $N=3000$ sites for various values of the forest birthrate $\bar{J}$. \Denis{We simulated the finite-size system for each value of $\bar{J}$, for various initial conditions with various initial grass cover proportions (20 initial conditions when the GKEs are monostable, and 90 initial conditions regularly sampled between $0.1$ to $0.9$ in the bistable case). There is remarkable agreement between the simulations of the Markov process and the bifurcation diagram in Figure \ref{two_state_expected_bifurcations}. Even after a reasonably long time (t = 100), the finite-size system remained away from the all grass state for many initial conditions when other stable attractors existed in the GKEs of the mean-field limit.} Moreover, matching the initial grass cover with the final state, we observed that the ``basin of attraction'' of each quasi-attractor for the transient states of the finite-size Markov process system shows excellent agreement with the unstable fixed point from the GKEs. Past the transcritical bifurcation we continued to observe absorption to the all-grass state, although this fixed point is unstable in the GKEs for this parameter range.  
	
	\begin{figure}
		\centering
		\includegraphics[width=\textwidth]{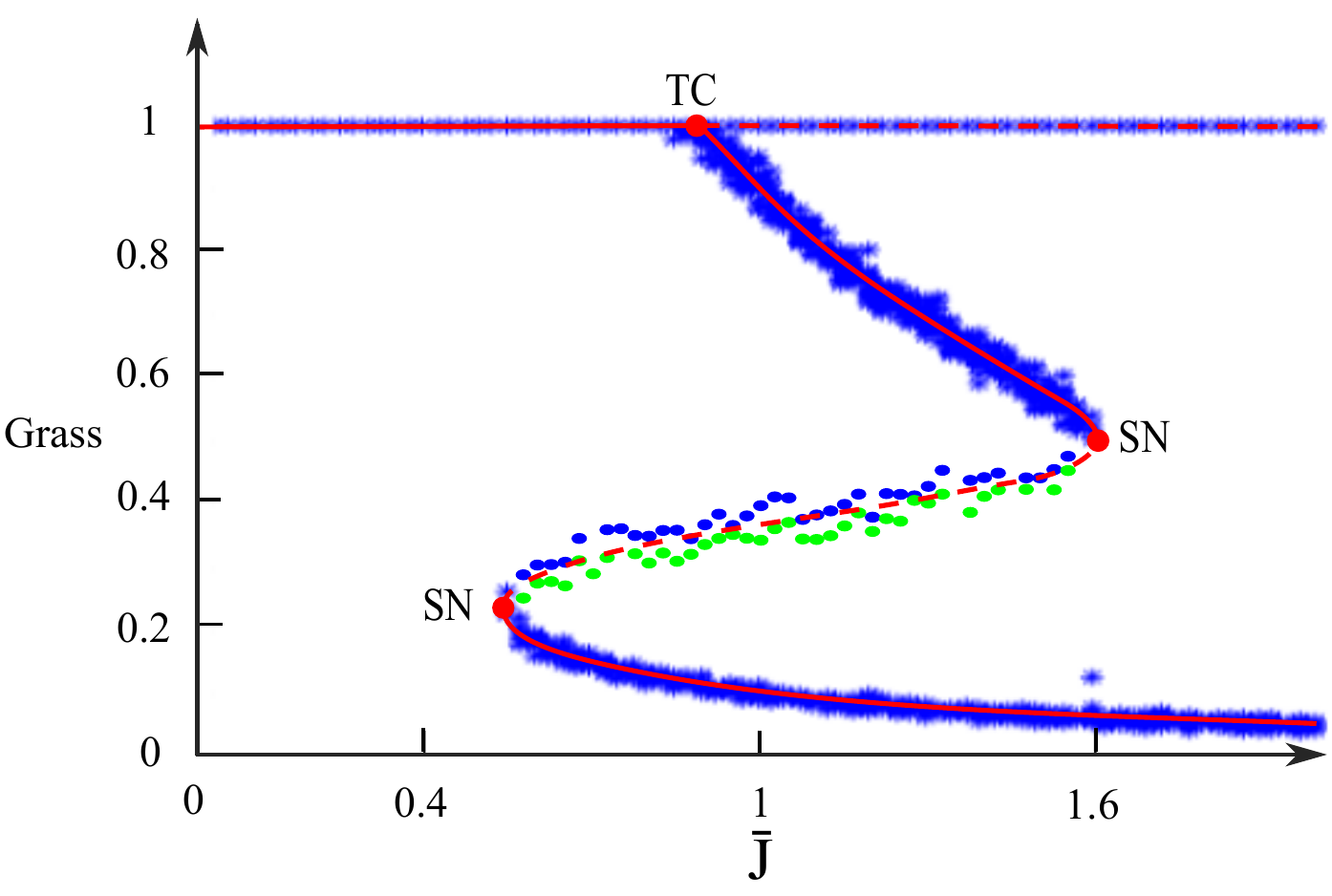}
		\caption{\textbf{Blue starred markers:} Grass proportion in the finite-size Staver-Levin model for time $t = 100$ with $N = 3000$ and $M=1$ (macroscale model). End states from $20$ different initial conditions are plotted for each value of the forest recruitment rate $\bar{J}$, and the attraction basin was computed using $90$ different initial conditions for more precision. \textbf{Red curves:} The bifurcation diagram of the GKEs for the corresponding mean-field limit is overlaid with solid red lines indicating stable equilibria and dashed red lines indicating unstable equilibria. \textbf{Blue/green dots:} Blue dots mark the lowest initial condition for the finite-size system that resulted in a final state on the upper equilibrium branch and green dots mark the largest average initial condition which gave an end state on the lower equilibrium branch. \textbf{Red dots:} These markers indicate the bifurcation points in the GKEs of the mean-field limit --- SN: saddle-node bifurcation, TC: transcritical bifurcation.}
		\label{two_state_expected_bifurcations}
	\end{figure}

	\subsubsection{Quasi-stationary Distribution and Absorption Rate}
	The apparently paradoxical persistence of transient behavior and striking agreement between the mean-field and finite-size system in Figure \ref{two_state_expected_bifurcations} can be better understood in terms of the quasi-stationary distribution (QSD) of the finite-size system. Fix $N\in\mathbb{N}$ and consider the Markov process $\tilde{X}$ which tracks the proportion of grass sites in the finite-size system so that the state space of $\tilde{X}$ is $S = \{1,\, (N-1)/N,\, \dots, 1/N,\, 0 \}$. The QSD is the stationary distribution of the system conditional on not being absorbed by the all grass state, or equivalently, it is the stationary distribution of the new process $\tilde{X}^*$ on the restricted state space $S^* = \{(N-1)/N,\, \dots, 1/N,\, 0 \}$.
	\begin{definition}
		A distribution $x \in \mathbb{R}^{N}$ is a QSD for the process $\tilde{X}$ if for each $t \geq 0$,
		\[
		\mathbb{P}_{x}\left[ \tilde{X}(t) = j\, \Big{|}\, T > t \right] = x_j, \quad j \in S^*
		\]
		where $T = \inf\{t \geq 0: \tilde{X}(t) = 1 \}$ and $\mathbb{P}_{x_0}$ denotes the probability measure associated with the process $\tilde{X}$ starting from the initial distribution $x_0$. 
	\end{definition}
	The generator matrix of the process $\tilde{X}$ has the form
	\[
	\begin{pmatrix}
	0 & \textbf{0}^T \\
	\textbf{a} & Q
	\end{pmatrix}
	\]
	where $\textbf{0}^T = \{0,\dots,0\} \in \mathbb{R^N}$, \[\textbf{a} = \left\{\tfrac{1}{N}\phi\left(\tfrac{N-1}{N}\right),0,\dots,0 \right\}^T \in \mathbb{R}^N,\]
	and $Q$ is the $N \times N$ matrix given by
	\[
	\begin{pmatrix}
	- \tfrac{1}{N}\phi\left(\tfrac{N-1}{N}\right) - \tfrac{\bar{J}}{N}\tfrac{N-1}{N} & \tfrac{\bar{J}}{N}\tfrac{N-1}{N} & 0 & & \dots & 0 \\
	\tfrac{2}{N}\phi\left(\tfrac{N-2}{N}\right) & -\tfrac{2}{N}\left(\phi\left(\tfrac{N-2}{N}\right) +\bar{J}\tfrac{N-2}{N}\right)  &  \tfrac{2\bar{J}}{N}\tfrac{N-2}{N} & 0 & \dots & 0 \\
	\vdots & \ddots & \ddots &  & & \vdots \\
	\vdots & \dots & \dots & \tfrac{N-1}{N}\phi\left(\tfrac{1}{N}\right)  & -\tfrac{N-1}{N}\left(\phi\left(\tfrac{1}{N}\right) +\tfrac{\bar{J}}{N}\right)   & \tfrac{(N-1)\bar{J}}{N}\tfrac{1}{N} \\
	0 & 0 & \dots & 0 & \phi(0) & -\phi(0)
	\end{pmatrix}
	\]
	\Denis{Since $\phi(0)>0$ and $\bar{J}>0$,  $\tilde{X}^*$ is recurrent on $S^*$ and hence the QSD exists and is unique~\cite{darroch1967quasi,norris1998markov}. The QSD $x$ associated with the process $\tilde{X}$ is given by 
		\[
		x\,Q = - \rho x, \quad x\, \textbf{1}^T = 1,
		\] 
		where $\rho$ is the dominant eigenvalue of $Q$ (i.e., eigenvalue with largest real part). In other words, the QSD is the normalized eigenvector associated with the principal eigenvalue of $Q$.}
	
	The QSD for the two-state finite-size Staver-Levin system given by \eqref{eq:TwoStatesOnePatchNet} is plotted in panel A1 of Figure \ref{fig.qsd} for each value of $\bar{J}$. The mass of the QSD concentrates on the stable states of the GKEs for the corresponding mean-field limit --- partly explaining the persistent transients observed in Figure \ref{two_state_expected_bifurcations}. A more complete explanation of this persistence is obtained by calculating the absorption rate for the Markov jump process $\tilde{X}$ - this is the speed at which we expect the particle system to approach the all grass state from an arbitrary initial distribution. More precisely, the \revision{dominant} eigenvalue of the sub-stochastic restricted transition matrix $Q$, denoted by $\rho$, gives the speed of approach to the all grass state in the sense that for any initial distribution $x_0$,
	\[
	\lim_{t\to\infty} \mathbb{P}_{x_0}\left[ \tilde{X}(t+s) \neq 1\, | \, T > t \right] = e^{-\rho s}, \quad s \geq 0.
	\]
	Panel A2 of Figure \ref{fig.qsd} shows the absorption rate $\rho$ as a function of $\bar{J}$ for multiple values of the system size $N$. The absorption rate shows a decreasing profile, indicating (as expected) a decay of the rate of absorption as the forest tree birth rate increases. More strikingly, the rates decay with a sharpening profile as $N$ increases, with a consistent switch from a rate around $0.1$ to an almost zero rate at $\bar{J} \approx 0.55$ --- coinciding exactly with the appearance of a saddle-node bifurcation in the GKEs of the mean-field system (cf. Figure \ref{two_state_expected_bifurcations}). Therefore, consistent with Figure \ref{two_state_expected_bifurcations}, we expect to observe persistent transient behavior matching the QSD for values of $\bar{J}$ past the first saddle-node bifurcation.
	\begin{figure}
		\centering
		\includegraphics[width=\textwidth]{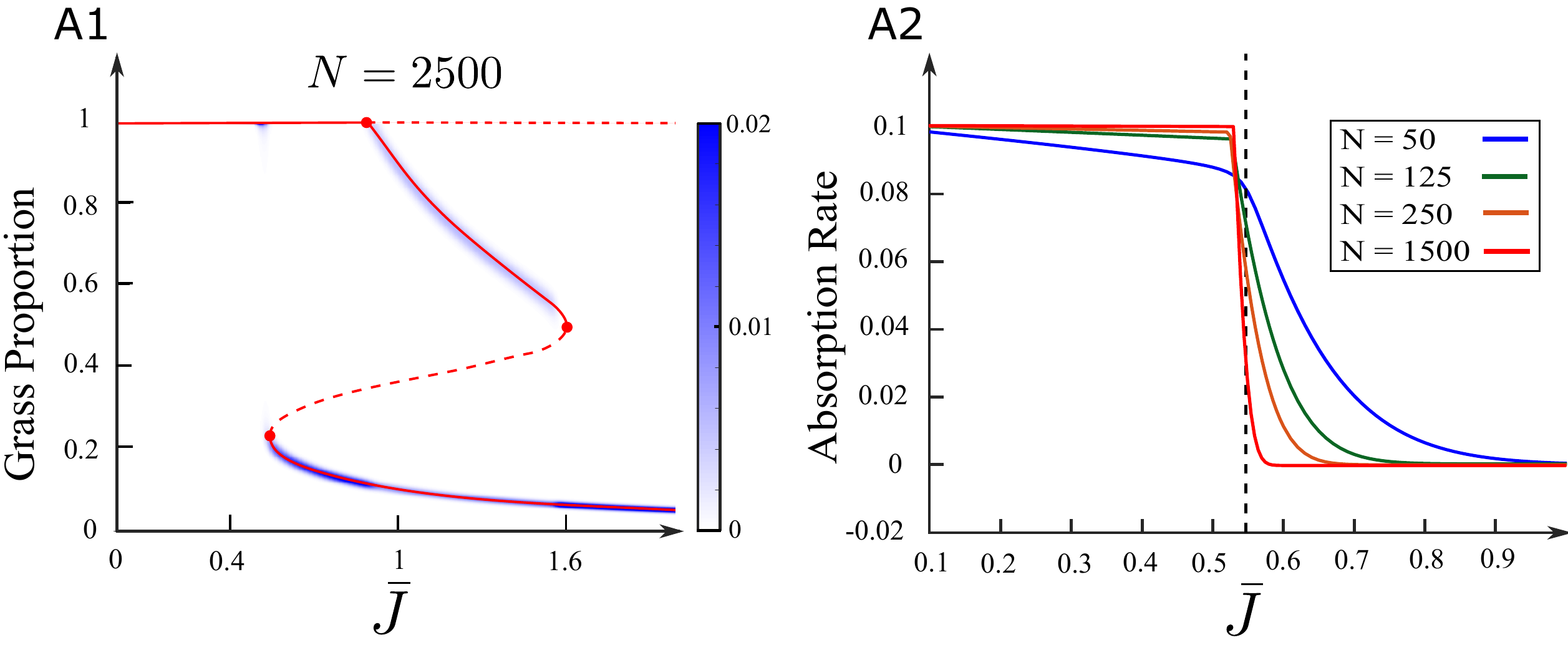}
		\caption{\textbf{A1:} Quasi-stationary distribution of the interacting particle system for $N = 2500$ for each value of $\bar{J}$.  \textbf{A2:} Absorption rate to the all-grass state as a function of $\bar{J}$.}
		\label{fig.qsd}
	\end{figure}
	\subsubsection{Stochastic Oscillations: Solutions with Time Periodic Law}
	The original Staver-Levin model \Denis{given by the system of nonlinear ODEs \eqref{eq.Staver_levin_ode} has regions of parameters where the only stable attractor is a periodic orbit. In this regime, a simple consequence of Theorem~\ref{thm_convergence_Kstates} is that the mean-field limiting process will exhibit stable oscillations.} This is a remarkable property and has previously been observed in continuous-time Markov processes described by Brownian driven stochastic differential equations (SDEs)~\cite{scheutzow1986periodic,touboul2012noise}. To the best of our knowledge, the mean-field process defined by \eqref{eq:MeanFieldMeso} is the only example identified to date of a McKean-Vlasov jump process which has a periodic law. Mathematical methods for studying nonstationary solutions are still to be developed for jump processes but in our case, the existence of a periodic law is based on the derivation of the GKEs and their bifurcation analysis, avoiding the need for probabilistic arguments. 
	
	\Denis{Panels A1 and A2 of Figure~\ref{fig.periodic_IPS} show the periodic solutions generated by the 4-species macroscale model introduced in Example \ref{eg.macroscale} in a periodic regime with $M = 1$. We observe, consistent with Theorem \ref{thm_convergence_Kstates}, that (for N sufficiently large)} trajectories of the finite-size Staver-Levin system remain close to the trajectories of the GKEs of the corresponding mean-field limit, a result analogous to the effect observed in \cite[Figure 4]{durrett1998spatial}. Periodic orbits in the original Staver--Levin model were shown to grow and disappear at a heteroclinic orbit when the forest tree birth rate increases~\cite{touboul2018complex}. We thus explored the dynamics of the finite-size Staver-Levin model as a function of the forest tree birthrate $\bar J$. The heteroclinic orbit in the original Staver--Levin model connects three fixed points: the all-grass equilibrium, a savanna equilibrium and a mixed-saddle equilibrium where all species are present. We simulated trajectories of the finite-size system for various values of $\bar J$ (see Figure \ref{fig.periodic_IPS} C2), and found that while transient trajectories show very similar dynamics to the GKEs of the mean-field limit, oscillations became transient for parameters too close to the heteroclinic cycle; the system is rapidly absorbed by the savanna subsystem (i.e., absence of forest trees), and reaches a fixed point on this subsystem. Rigorously, this fixed point is unstable for the GKEs, because of an invasion of forest trees, and trajectories of the (deterministic) GKEs visit regions very close to that fixed point. The finite-size system, following closely these trajectories, thus reaches states with very low numbers of forest trees where extinction of forest becomes very likely. Near heteroclinic cycles the finite-size system is thus vulnerable to fluctuations in the vicinity of absorbing states. Interestingly, this effect of absorption near the heteroclinic cycle shows dynamics \Denis{significantly different from simulations of the GKEs with Brownian noise (which lead to stochastic resonance phenomena)}~\cite{touboul2018complex}. From an ecological viewpoint, this reveals an interesting fragility of the ecosystems when trajectories approach absorbing subsystems.
	\begin{figure}
		\centering
		\includegraphics[width=\textwidth]{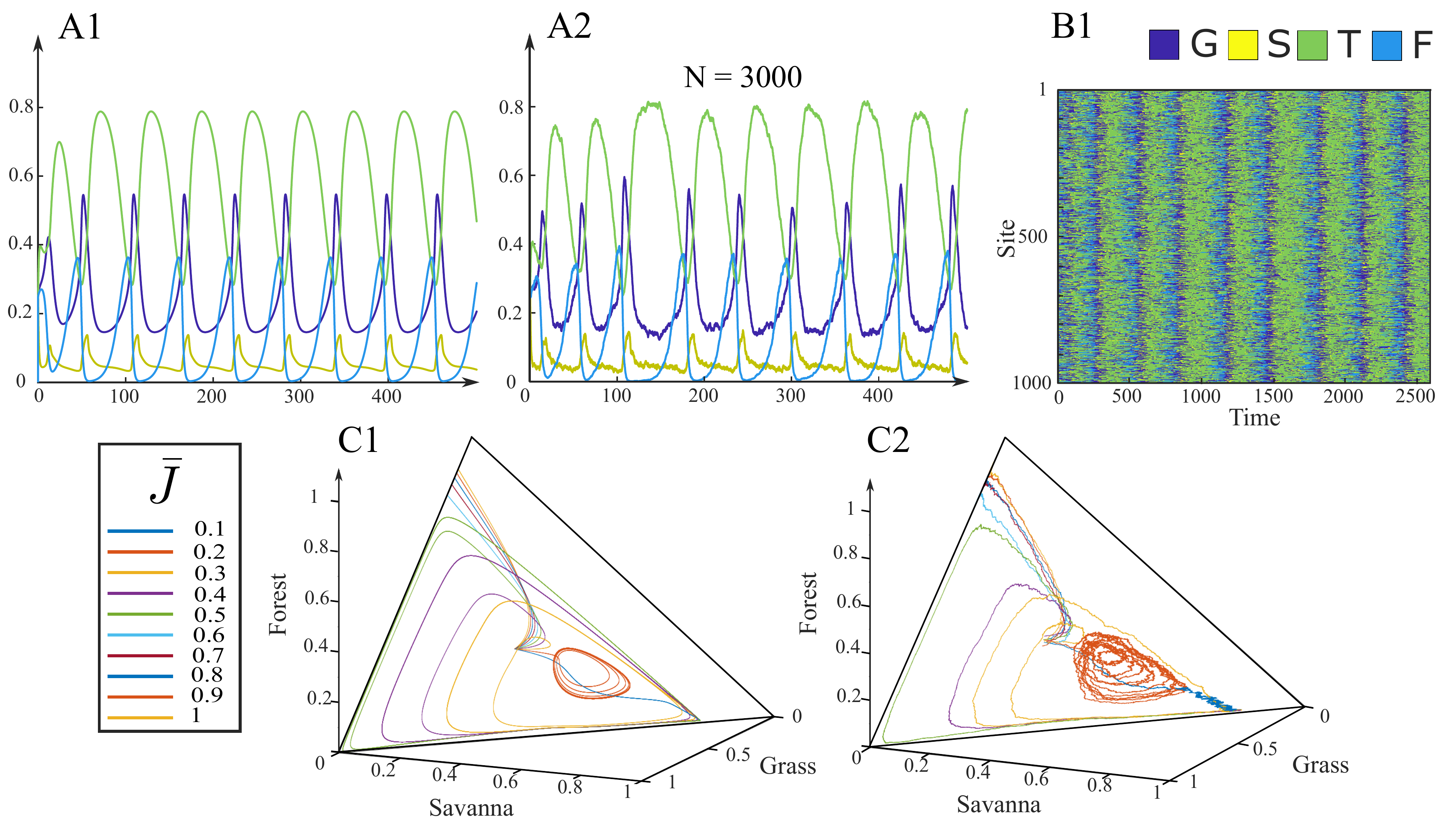}
		\caption{\textbf{A1/A2:} Comparison of the single patch macroscopic particle system model for $N = 3000$ (A2) and the solution of the corresponding Kolmogorov equations in a periodic parameter regime. \textbf{B1:}  Simulation of the macroscopic model for a large number of sites in a periodic regime. \textbf{C1/C2:} Comparison of trajectories of the particle system and Kolmogorov equation solutions in the state-space for various values of $\alpha$.}\label{fig.periodic_IPS} %Parameters: $\mu=0.1$, $\nu=0.05$, $\bar{J}=0.25$ and $\beta = 0.4$.}
	\end{figure}
	
	\subsection{Mesoscopic Model and Spatially Extended Staver-Levin Model}
	
	In our mesoscopic framework, we derived a mean-field \revision{spatially extended} jump process characterizing the dynamics of sites distributed in space. The distribution of this process is described by integro-differential equations (IDEs) which characterize the spatio-temporal dynamics of vegetation at this scale (the GKEs). The rigorous derivation of this model from first principles provides us with a new model that incorporates heterogeneity in the sites density via the choice of the initial site distribution $q$. The study of this system can yield valuable information on the distribution of vegetation in space but we defer a full investigation of this models dynamics to further work. Here, we concentrate solely on convergence properties and dynamical consistency between the mesoscopic finite-size Staver-Levin model and the GKEs of the corresponding mean-field limit. 
	
	We focus on the simplest subsystem possible, the forest-grass subsystem, in order to illustrate our convergence results, while avoiding discussion of the more complex dynamics of the full system. In this subsystem, the dynamics are fully described by the fraction of forest trees at a given location $r\in \Gamma$ and time $t\in\R^+$, i.e. the quantity
	\[P_F(t,r) :=  \mathbb{P}\left[ \bar{X}(t,r)  = F \right].\]
	The GKEs of the mean-field process are thus given by the following nonlocal IDE:
	\begin{multline}\label{eq.kolmogorov_spatial}
	\frac{\partial}{\partial t}P_F(t,r) = P_G(t,r) \,\int_{\Gamma} J(r',r)\, P_F(t,r') \,dq\left(r' \right)\\ - P_F(t,r) \,\phi\left( \int_{\Gamma} W(r',r)\,P_G(t,r')\, dq\left(r' \right) \right), \quad (t, r)\in \mathbb{R}_+ \times \Gamma,
	\end{multline}
	with $P_G(r,t)=1-P_F(r,t)$. We study two ecologically  relevant behaviors arising in these systems, waves of invasion and emergence of fronts in heterogeneous landscapes, as well as the consistency between solutions of the finite-size system and the integro-differential equation~\eqref{eq.kolmogorov_spatial}.
	
	\subsubsection{Waves of invasion}\label{sec.waves}
	Consider a homogeneous one-dimensional\\ landscape, assumed for simplicity to be a ring (i.e. $\Gamma=\mathcal{S}_L$ the one-dimensional torus of length $L$, represented by the interval $[0,L]$ with the boundaries identified). Assuming a homogeneous landscape amounts to considering a uniform site density $q$ on $\Gamma$ within our framework. Both seed dispersal and fire propagation kernels are Gaussian functions of the form
	\[W(r,r')=\frac{J(r,r')}{\bar{J}} = \frac{C(\sigma)}{\sigma\sqrt{2\pi}} e^{-\frac{d_{\Gamma}(r,r')^2}{2\sigma^2}}, \quad (r,r')\in\Gamma^2\]
	where $d_{\Gamma}(r,r')$ denotes the distance between the points $r$ and $r'$ on the ring $\Gamma$. The nonstandard normalization factor $C(\sigma)=L \left(2\Phi(L/2\sigma)-1 \right)^{-1}$ is due to the fact that our Gaussian kernels are compactly supported and being integrated against the uniform measure on $[0,L]$\footnote{$\Phi$ denotes the cumulative distribution function of a standard Normal random variable here.}. In the present setup, spatially homogeneous solutions of the IDE \eqref{eq.kolmogorov_spatial} solve the original Staver-Levin ODEs~\eqref{eq.Staver_levin_ode} due to the translation invariance \revision{on the ring}. Therefore, spatially homogeneous stationary solutions are given by the bifurcation diagram of the ODE  in Fig.~\ref{two_state_expected_bifurcations}, but spatial interactions may alter the stability of these steady  states. Grassland is the unique spatially homogeneous equilibrium when the forest-tree birth rate $\bar{J}$ is low enough. Forest is the unique spatially homogeneous steady state for sufficiently large $\bar{J}$, while multiple homogeneous steady states co-exist for intermediate values of $\bar{J}$. We studied the competition between forest  and  grass in the coexistence regime by choosing an initial state with a region of forest trees flanked by grass on either side. In Figure \ref{fig.maxwell} (panels A1, A2, A4 and A5) we observe that, depending on the forest-tree birth-rate $\bar{J}$, grass may invade the patch of forest trees or trees may invade the grassy regions (in both the finite-size system and the GKEs). There is excellent agreement between solutions of the GKEs of the mean-field limit (i.e. the IDEs given by \eqref{eq.kolmogorov_spatial}) and those of the finite-size Markov process. 
	
	\subsubsection{Forest-grass fronts in Heterogeneous Landscapes}
	One advantage of our framework is the ability to derive mesoscopic models which incorporate environmental heterogeneity by choosing a nonuniform initial site distribution. For instance, ecologically, soils may have substantial effects on tree establishment potential, which can be reflected in our model by via lower site density regions. Reduced site density induces two opposite effects: a lessened ability to carry fires, but also lower opportunity for trees to grow. The results of simulations of the finite-size model, as well as the GKEs of the corresponding mean-field model, with a variable density of sites in the forest-grass subsystem are presented in Fig.~\ref{fig.maxwell}, panels B1--B3. For these simulations,  the site density was chosen to be a trapezoid on $\Gamma=[0,L]$, i.e. \[dq(x)=\left(a+b\,x\right)\mathbbm{1}_{[0,L]}(x)\,dx,\quad a,b>0.\] This choice ensures an increasing density of sites along the interval $\Gamma$: regions near $x=0$ have lower densities than regions near $x=L$. Space-time plots of the solutions to the finite-size model and the GKEs (panels B1 and B2) illustrate that lower site density favors forest, while higher densities favor grassland. At intermediate site densities, a sharp front forms between forest and grassland as neither species is able to invade the other, akin to a Maxwell point.  This type of solution is also referred to as a ``front pinning phenomenon'' as the sharp front forms because the wave speed of the wave of invasion which would typically annihilate the less competitive species (in the corresponding homogeneous domain problem) approaches zero at the Maxwell point \cite{van2015resilience,wuyts2018fronts}. Panel B3 compares the bifurcation diagram of the appropriate GKEs without spatial interaction ($\sigma\to 0^+$ in the kernels) with the final solution profiles of the finite-size model and the GKEs. We observe that the solution of the GKEs (in black) essentially interpolates between the two stable equilibria of the GKEs without spatial interaction (solid red). The solution of the finite-size model approximately matches that of the GKEs but naturally has some stochastic excursions since we are observing a single realization of the process.
	\begin{figure}
		\centering
		\includegraphics[width=\textwidth]{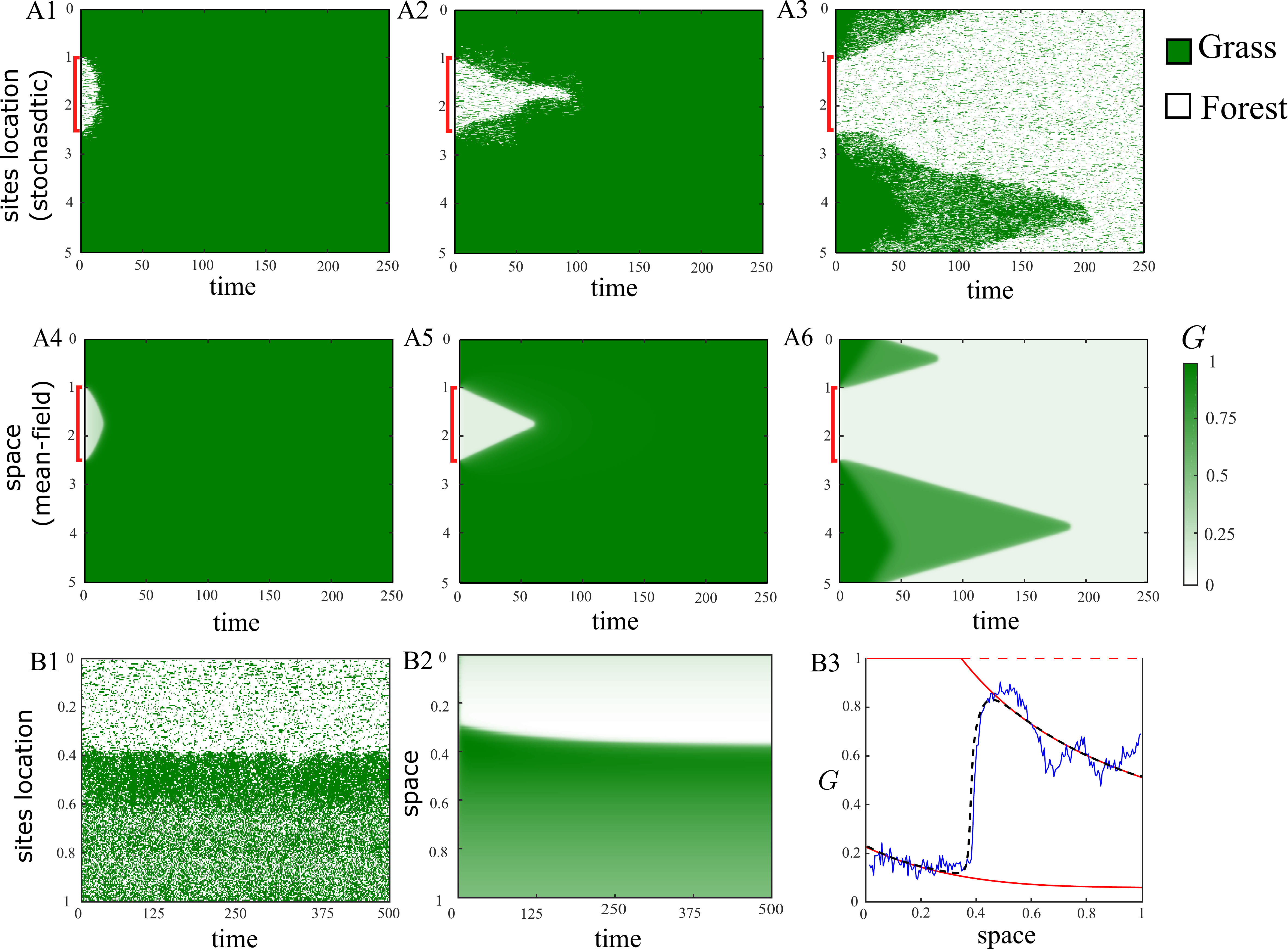}
		\caption{\textbf{A1--A6:} Comparison of solutions of the finite-size stochastic system and the corresponding GKEs for a grass dominant regime (A1, A2, A4 and A5) and a regime in which forest invades grass (A3 and A6). \textbf{Parameters:} $\Gamma = \mathcal{S}_5$ the one-dimensional torus of  length $5$, $N = 3000$, $\bar{J} = 0.5$ (A1, A4), $\bar{J}=0.9$ (A2, A5) and $\bar{J}=1.25$ (A3, A6), $\sigma = 0.05$. Initial conditions are the same for all simulations with a block of forest on $[1,\,2.5]$ (marked in red). The colorbar for A4-A6 is shown to the side. Sites are either 1 (grass) or 0 (forest) in A1-A3 and hence we use the same colorbar.\vspace*{3pt} \newline \textbf{B1--B3:} Simulated example of front pinning/Maxwell point phenomenon for the mesoscale model. Panel B1 shows a single realization of the finite-size model with panel B2 showing the solution to the corresponding GKEs. B3 shows the final solution profile of the GKEs (dashed black line), bifurcation diagram for the GKEs without spatial interaction (stable equilibria in solid red, unstable equilibria in dashed red), and the time averaged profile of the solution to the particle system (solid blue). \textbf{Parameters:} $\Gamma = [0,1]$, $N = 2000$, $\bar{J} = 1.1$, $\sigma = 0.02$, reflecting boundary conditions.}\label{fig.maxwell}
	\end{figure}
	
	\section{Proofs of Main Results} \label{sec.proofs}
	In this section, we provide the details of the proof of Theorem~\ref{thm_convergence_Kstates}. To this end, we use a bijection between the original state space $S^K$ and the $K$ vectors of the canonical basis in $\R^K$, and reformulate the finite-size Markov process as a stochastic differential equation. This allows a simpler analysis of all models in the class, to prove existence and uniqueness of solutions of the mean-field limit, as well as convergence of the finite system towards the limit system as the system size diverges.

	\subsection{Mathematical Preliminaries}\label{sec:setting}
	Throughout and in the proofs which follow, we work on a complete probability space $\left(\Omega,\mathcal{F},\mathbb{P}\right)$ endowed with a countable family of i.i.d Poisson point processes \Denis{$\mathcal{N} = \{\mathcal{N}^{i}_{x,y}(t): \, (x,y) \in S^K \times S^K, \, i \in \mathbb{N}, \, t \geq 0\}$ on $\mathbb{R}^2_+$ with compensators (i.e. predictable part) given by the Lebesgue measure on $\mathbb{R}_+^2$ with $\mathbb{R}_+:=[0,\infty)$. Define the filtration 
		\begin{equation}\label{eq.filtration}
		\mathcal{F}_t = \sigma\left( \mathcal{N}^{i}_{x,y}(A\times B): \, (x,y) \in S^K \times S^K, \, i \in \mathbb{N}, \, A \in\mathcal{B}(\mathbb{R}_+),\, B \in \mathcal{B}([0,t])\right ), \,\, t \geq 0,
		\end{equation}
		where $\mathcal{B}(H)$ denotes the space of Borel sets of $H$ and $\mathcal{N}^{i}_{x,y}(A\times B)$ is the number of points of the point process $\mathcal{N}^{i}_{x,y}$ in $A\times B$.} \revision{$\mathcal{D}(\mathbb{R}_+;E)$ denotes the Skorohod space of c\` adl\` ag functions with values in the space $E$ (typically, considered here to be $\R_+$ or $\R^p$ for some $p\in\N$), i.e. functions that are right-continuous functions with left limits everywhere}. We recall the following useful lemma (see Graham and Robert \cite{graham2009interacting}) which applies readily to processes in $\mathcal{D}([0,T];S^K)$ for our purposes. 
	\begin{lemma}\label{F_t_martingale}
		Suppose that the processes $Y = \{Y(t): t \geq 0\}$ and $Z = \{Z(t): t \geq 0\}$ are in $\mathcal{D}(\mathbb{R}_+;\mathbb{R}_+)$ and are adapted to $\left(\mathcal{F}_t\right)_{t \geq 0}$. If $\mathcal{N}$ is a Poisson process on $\mathbb{R}^2_+$ with compensator given by the Lebesgue measure on $\mathbb{R}_+^2$, then the process $I = \{I(t): t\geq 0\}$ given by
		\[
		I(t) = \int_0^t \int_0^\infty Y(s^-)  \1_{\left\{ 0 \leq z \leq Z(s^-) \right\}}\left[\mathcal{N}(dz,ds) - dz\,ds\right],
		\]
		is a local $\mathcal{F}_t$-martingale. Moreover, $I(t)$ is a piecewise constant process with jumps of time-dependent size (i.e. $Y(t)$ is the size of a jump at time $t$) occurring at the times of the jumps of a non-homogeneous Poisson process with instantaneous rate $Z(t)$. 
	\end{lemma}
	In particular, under the hypotheses and notation of Lemma \ref{F_t_martingale}, we have
	\[
	\Denis{\mathbb{E}\left[ \int_0^t \int_0^\infty Y(s^-) \1_{\left\{ 0 \leq z \leq Z(s^-) \right\}}\mathcal{N}(dz,ds) \right] = \mathbb{E}\left[ \int_0^t Y(s^-) Z(s^-)\,ds \right],}
	\]
	a fact which is used frequently in the arguments which follow.
	
	In order to prove that the mean-field equation introduced in Theorem~\ref{thm_convergence_Kstates} admits a well-defined solution, we must consider spatially extended discrete-state continuous-time stochastic processes. $\Gamma$ always denotes a Borel set in $\mathbb{R}^2$ and $q$ is a probability measure on $(\Gamma, \mathcal{B}(\Gamma))$. To analyze these processes, we work in the space $\mathcal{M}_T$ of c\'adl\`ag processes $Y= \{Y(t,r): \, r \in \Gamma, \, t \in [0,T]  \} $ with state space $S^K$ which are measurable on the product space $(\Omega \times \Gamma)$ and, for each fixed $r\in \Gamma$, $Y(t,r)$ is $\mathcal{F}_t$-adapted\footnote{Rigorously, these processes and their measurability are defined through Markov kernels as done in the statement of assumption $H_{IC}$. Here, measurability is stated to make sure the norm is well-defined; the measurability and adaptedness will be obvious for the processes relevant to our proof, and therefore the norm $\|\cdot\|_{\mathcal{M}_T}$ is well-defined.}. Define the norm $\| \cdot \|_{\mathcal{M}_T}$ on $\mathcal{M}_T$ by {
	\begin{equation}\label{def_norm}
	\|Y\|_{\mathcal{M}_T} = \mathcal{E}\left[ \mathbb{E}\left[ \sup_{t\in [0,T]} \|Y(t,r')\| \right] \right] = \int_\Gamma \mathbb{E}\left[  \sup_{t\in[0,T]}\|Y(t,r') \|\right]  dq(r')
	\end{equation}
	where $||\cdot||$ denotes the standard Euclidean norm on $\mathbb{R}^K$. The norm $\| \cdot \|_{\mathcal{M}_T}$ identifies processes in $\mathcal{M}_T$ that are $\mathbb{P}$ a.s.-$q$-a.e. equal on $[0,T]$.} The space $\mathcal{M}_T$ is a complete separable metric space under this norm (see, e.g.,~\cite[Chapter 3]{billingsley2013convergence}).
	\subsection{Stochastic differential equation formulation of the Markov process}
	Theorem~\ref{thm_convergence_Kstates} deals with an interacting Markov process model with finite state space $S^K$ containing $K$ elements $(x_1,\dots,x_K)$. For convenience and definiteness, we assume (up to bijection) that the elements of $S^K$ are given by the canonical basis of $\R^K$:
	\[
	x_1 := \tfrac{1}{\sqrt{2}}\left(1,0,\dots,0\right)^T, \quad x_2 := \tfrac{1}{\sqrt{2}}\left(0,1,0,\dots,0\right)^T,\dots, \quad x_K := \tfrac{1}{\sqrt{2}}\left(0,\dots,0,1\right)^T.
	\]	
	This choice (and scaling) is particularly convenient in that it allows using the usual Euclidean norm on $\mathbb{R}^K$ as our norm on $S^K$, and in that norm, all states are equidistant (one unit apart). 
	
	\revision{Using that representation, the state of site $i$ is a process in $\mathcal{D}(\R_+,S^K)$ whose evolution is described by the Poisson-driven stochastic differential equation}:
	\begin{multline}\label{eq.network_4_state}
	X^{i}(t) = X^{i}(0) + \\ \sum_{y \in S^K}\int_0^t (y-X^i(s^{-})) \sum_{\substack{x \in S^K, \\ x \neq y} }\1_{ \{X^i(s^-) = x\} } \int_0^\infty \1_{\{0\leq z\leq R_{x,y}^{i,N}(X(s^{-}))\}} \mathcal{N}^{i}_{x,y}(dz,ds)
	\end{multline}
	\john{for each $t \geq 0$. The transition rate intensity $R_{x,y}^{i,N}$ is given by equation \eqref{eq:RatesGeneral} and we write $R_{x,y}^{i,N}(X(s^{-}))$ to emphasize that the transition rate depends on the entire system $X(s^-) = \left\{X^i(s^-): i \in \{1,\dots,N\} \right\}$. In equation \eqref{eq.network_4_state}, for each $x,y\in S^K \times S^K$ and $i\in\{1,\dots, N\}$, each $\mathcal{N}^{i}_{x,y}$ is an independent Poisson point process on $\mathbb{R}^2_+$ with compensator (or predictable part) given by the Lebesgue measure on $\mathbb{R}_+^2$ (see Section~\ref{sec:setting}). Indeed, the processes \[\int_0^{\infty} \1_{\{0 \leq z \leq f(t)\}} \mathcal{N}^{i}_{x,y}(dz,ds)\] are inhomogeneous Poisson processes with instantaneous rate $f(t)$, and therefore the system of SDEs defined by equation  \eqref{eq.network_4_state} is a Markov process with independent exponentially distributed transitions with rates $R_{x,y}^{i,N}(X(s^{-}))$.}
	
	Under assumptions H1 and H2, we can define the following quantities:
	\begin{enumerate}[(i.)]
		\item There exists $L>0$ such that for all $x,y \in S^K$ and all $X,Y \in R_+$,
		\[\lvert \Phi_{x,y}(X) - \Phi_{x,y}(Y) \rvert \leq L \,\lvert X- Y \rvert.\]
		\item $\| \Phi_{x,y} \|_\infty := \sup_{u\in [0,\,\|W\|_\infty]} \left|\Phi_{x,y}(u)\right|<\infty$.
		\item $\| W_{x,y} \|_\infty := \sup_{r,r'} \left| W_{x,y}(r,r')\right|<\infty$.
		\item $\| W \|_\infty := \sup_{x,y \in S^K}\| W_{x,y} \|_\infty<\infty$.
	\end{enumerate} 
	Under the assumptions H1, H2 and $H_{IC}$, strong existence and uniqueness of solutions for the finite-size Markov jump process is classical (see, e.g., \cite{ikeda2014stochastic}).
	\subsection{The McKean-Vlasov process: Definition, existence and uniqueness of solutions}\label{sec:existUnique}
	Theorem~\ref{thm_convergence_Kstates} states that the finite-size Markov process with transition rates given by \eqref{eq:RatesGeneral} converges in distribution towards a \revision{spatially extended} process defined on the support of $q$ on $\Gamma$. The limit process, denoted by $\bar{X}$ and referred to as the mean-field limit, has inter-jump times that are independent and exponentially distributed. The jump rates of $\bar{X}$ at time $t$ depend self-consistently on the law of the process itself and are given by:
	\begin{equation}\label{eq:Rate2}
	x \to y \qquad \Phi_{x,y}\left( \int_\Gamma W_{x,\,y} \left(r,\,r'\right)\mathbb{P}\left[\bar{X}(t^-,r') = \psi(x,\,y) \right]dq(r') \right) =: \bar{R}_{x,y}\left(\bar{X}(t^-),r \right),
	\end{equation}	
	for each pair of states $(x,y) \in S^K \times S^K$ and where $\bar{X}(t^-) = \left\{\bar{X}(t^-,r') : \, r' \in \Gamma \right\}$. $\bar{X}$ is non-Markovian because  of the dependence of the rates upon the law of the solution, and not upon the state of the system. The mean-field limit described above is a so-called McKean-Vlasov process~\cite{villani2002review}. Due to the nonstandard nature of these processes, classical existence and uniqueness theory is not sufficient to \revision{establish} well-posedness. Although the constructions are standard, we provide the necessary existence and uniqueness arguments to show that the mean-field limit is well-posed in order to keep the presentation self contained.
	
	\begin{definition}\label{def_strong_solution}
		A strong solution of the mean-field process with initial condition $\xi_0(r)$ with a regular distribution in space (in the sense of hypothesis $H_{IC}$) is a process $\bar{X}\in \mathcal{M}_T$ such that for $q$-almost every $r$, $\bar{X}(0,r)$ is $\mathbb{P}$-almost surely equal to $\xi_0(r)$ and, given the law of the solution up to time $t$, the process jumps from state $x$ to a state $y$ at a rate $\bar{R}_{x,y}$ given by equation \eqref{eq:Rate2}. 
		
		Equivalently, given a family of Poisson processes $\mathcal{N} = \{\mathcal{N}_{x,y}(t): \, (x,y) \in S^K \times S^K,\, t \geq 0 \}$, a solution $\bar{X}\in \mathcal{M}_T$ is a process such that for $q$-almost every $r$, $\bar{X}(0,r)$ is $\mathbb{P}$-almost surely equal to $\xi_0(r)$ and, for $q$-almost every $r \in\Gamma$ and each $t\in[0,T]$, the following  equation holds $\mathbb{P}$-almost surely:
		\begin{multline}\label{eq.mean_field_existence}
		\bar{X}(t,r) = \xi_0(r) + \\ \sum_{y\in S^K} \int_0^t (y-\bar{X}(s^{-},r)) \sum_{\substack{x \in S^K, \\ x \neq y} }\1_{ \{ \bar{X}(s^-,r) = x\} } \int_0^\infty \1_{\left\{0\leq z\leq \bar{R}_{x,y}\left(\bar{X}(s^-),r \right) \right\}} \mathcal{N}_{x,y}(dz,ds),
		\end{multline}
	\end{definition}
	
	The equivalence between the process associated with transition rates~\eqref{eq:Rate2} and the process which solves the mean-field equation~\eqref{eq.mean_field_existence} follows from the same classical arguments as those allowing us to write up the Poisson driven SDEs describing the evolution of the finite-size system. Indeed, if the solution to~\eqref{eq.mean_field_existence} exists and assuming that $\bar{X}(t^-,r)=x$, then a jump of size $(y-x)$ (therefore, moving the process from $x$ to $y=(y-x)+x$) arises at a random, exponentially distributed time with rate $\bar{R}_{x,y}\left(\bar{X}(t^-),\,r \right)$.  
	
	\begin{definition}\label{def_unique}
		The solution to \eqref{eq.mean_field_existence} is unique if for any two strong solutions $\bar{X}$ and $\tilde{X}$ to \eqref{eq.mean_field_existence}, the event
		\[
		\left\{ \bar{X}(t,r)  = \tilde{X}(t,r) \mbox{ for each } t \in [0,T] \right\}
		\]
		has probability one with respect to the product measure $\mathbb{P} \bigotimes q$.
	\end{definition}
	
	\begin{theorem}\label{thm.exist_unique_spatial}
		Suppose H1 and H2 hold. Let $\xi_0(r)$ be a spatially measurable random variable on $S^K$ (in the sense of $H_{IC}$) independent of the Poisson processes $\mathcal{N}_{x,y}$ for $x,y \in S^K$. The mean-field equation \eqref{eq.mean_field_existence} with initial condition $\xi_0(r)$ has a unique strong solution in $\mathcal{M}_T$.
	\end{theorem}
	\begin{proof}
		Define the mapping $\Psi$ which acts on processes $X = \{ X(t,r):\, r\in\Gamma,\, t \in [0,T] \}$ in $\mathcal{M}_T$ according to
		\begin{multline*}
		\Denis{\Psi\left( X \right)(t,r) =
			\xi_0(r) +} \\ \Denis{ \sum_{y\in S^K} \int_0^t (y-X(s^{-},r)) \sum_{\substack{x \in S^K, \\ x \neq y} }\1_{ \{ X(s^-,r) = x\} } \int_0^\infty \1_{\left\{0\leq z\leq \bar{R}_{x,y}\left(X(s^-),r \right) \right\}} \mathcal{N}_{x,y}(dz,ds)},
		\end{multline*}
		for each $(t,r)\in\mathbb{R}_+ \times\Gamma$. 
		
		According to Definition~\ref{def_strong_solution}, a solution to the mean-field equation~\eqref{eq.mean_field_existence} is a fixed point of $\Psi$. Therefore,  proving the theorem amounts to showing (i.) $\Psi$ has a fixed point, and (ii.) the solution obtained as the fixed point of $\Psi$ is unique according to Definition~\ref{def_unique}. Because of the consistency between Definition~\ref{def_unique} and the norm $\| \cdot \|_{\mathcal{M}_T}$ given by~\eqref{def_norm}, uniqueness will follow if all solutions are indistinguishable under $\| \cdot \|_{\mathcal{M}_T}$. We demonstrate both properties using a classical argument based on Picard's iterates. In order to be able to iterate the map $\Psi$, we must first establish that for any $X \in \mathcal{M}_T$, the process $Y := \Psi\left( X \right)$ also belongs to $\mathcal{M}_T$. For $X \in \mathcal{M}_T$, the integrals over $\Gamma$ in the jump intensity functionals
		\begin{align*}
		\Denis{\bar{R}_{x,y}\left(X(s^-),r \right) = \Phi_{x,y}\left( \int_\Gamma W_{x,y} \left(r,\,r'\right)\mathbb{P}\left[X(s^-,r') = \psi(x,y) \right]dq(r') \right),\quad y \in S^K,}
		\end{align*}
		are $\mathcal{F}_s$-measurable, continuous with respect to $r$ and finite due to the\\ $\left(\Omega \times \Gamma \right)$-measurability and boundedness of $X$, and regularity and boundedness of the $W_{x,y}$'s and $\Phi_{x,y}$'s respectively (see H1 and H2). Therefore, each $Y(t,r)$ is $\mathcal{F}_t$-adapted and measurable with respect to $\left(\Omega \times \Gamma\right)$. Since $X(t,r)\in S^K$, $Y(t,r)$ takes values in $S^K$ as well. Therefore, $\Psi$ maps $\mathcal{M}_T$ to itself. We can thus define the sequence of processes $(X^k(r))_{k \geq 0}$ as follows: 
		\[
		X^0 = \left\{ X^0(t,r): \, X^0(t,r) = \xi_0(r),\, r \in \Gamma, \, t \in [0,T] \right\}, 
		\]
		with the \Denis{$S^K$}-valued random variable $\xi_0(r)$ chosen to be $ \left(\Omega\times\Gamma\right)\mbox{-measurable}$
		and 
		\[
		X^{k+1} = \left\{ X^{k+1}(t,r): \, X^{k+1}(t,r) = \Psi\left( X^k \right)(t,r),\, r \in \Gamma, \, t \in [0,T] \right\}, \quad k \geq 1.
		\]
		For $k \geq 1$, estimate $\left|\left| X^{k+1}-X^k \right|\right|_{\mathcal{M}_T}$ as follows:
		\begin{align}\label{eq.exposition}
			\| X^{k+1}-X^k \|_{\mathcal{M}_T} &= \mathcal{E}\left[\mathbb{E}\left[ \sup_{t \in [0,T]}\| X^{k+1}(t,r)-X^{k}(t,r) \| \right]\right]\nonumber \\
			&= \sum_{y\in S^K} \sum_{\substack{x \in S^K, \\ x \neq y}} \mathcal{E}\left[\mathbb{E}\left[ \sup_{t \in [0,T]}\left\|\int_0^t \int_0^\infty \Delta(z,s) \, \mathcal{N}_{x,y}(dz,ds) \right\|\right]\right],
		\end{align}
		where 
		\begin{multline*}
			\Delta(z,s) = (y-X^k(s^-,r))\1_{\left\{ X^{k}(s^-,r) = x \right\}}\1_{\left\{0\leq z\leq \bar{R}_{x,y}\left(X^k(s^-),r \right) \right\}}\\ - (y-X^{k-1}(s^-,r))\1_{\left\{ X^{k-1}(s^-,r) = x \right\}}\1_{\left\{0\leq z\leq \bar{R}_{x,y}\left(X^{k-1}(s^-),r \right) \right\}}.
		\end{multline*}
	Thus
	\begin{align}\label{eq.inequality_1}
		\| X^{k+1}-X^k \|_{\mathcal{M}_T} &\leq \sum_{y\in S^K} \sum_{\substack{x \in S^K, \\ x \neq y}} \mathcal{E}\left[\mathbb{E}\left[ \sup_{t \in [0,T]}\int_0^t \int_0^\infty \|\Delta(z,s)\| \, \mathcal{N}_{x,y}(dz,ds) \right]\right]\nonumber \\
		&\leq \sum_{y\in S^K} \sum_{\substack{x \in S^K, \\ x \neq y}} \mathcal{E}\left[\mathbb{E}\left[ \int_0^T \int_0^\infty \|\Delta(z,s)\| \, \mathcal{N}_{x,y}(dz,ds) \right]\right],
	\end{align}
	where the second inequality follows from the fact that for each pair $(x,y) \in S^K\times S^K$, $\mathcal{N}_{x,y}$ is a nonnegative random measure on $\mathbb{R}_+^2$; this step only requires boundedness of $\Delta$ on $\mathbb{R}_+^2$ and does not depend on the regularity of the integrand. From here, estimate $\|\Delta(z,s)\|$ by breaking it into 3 parts as follows:
	\begin{align*}
		\|\Delta(z,s)\| &= 	\|(y-X^k(s^-,r))\1_{\left\{ X^{k}(s^-,r) = x \right\}}\1_{\left\{0\leq z\leq \bar{R}_{x,y}\left(X^k(s^-),r \right) \right\}}\\
		&\qquad\pm (y-X^k(s^-,r))\1_{\left\{ X^{k-1}(s^-,r) = x \right\}}\1_{\left\{0\leq z\leq \bar{R}_{x,y}\left(X^{k-1}(s^-),r \right) \right\}}\\
		&\qquad\pm (y-X^k(s^-,r))\1_{\left\{ X^{k}(s^-,r) = x \right\}}\1_{\left\{0\leq z\leq \bar{R}_{x,y}\left(X^{k-1}(s^-),r \right) \right\}}\\ &\qquad- (y-X^{k-1}(s^-,r))\1_{\left\{ X^{k-1}(s^-,r) = x \right\}}\1_{\left\{0\leq z\leq \bar{R}_{x,y}\left(X^{k-1}(s^-),r \right) \right\}}\|\\
		&\leq \|y-X^k(s^-,r) \|\left|\1_{\left\{0\leq z\leq \bar{R}_{x,y}\left(X^k(s^-),r \right) \right\}} - \1_{\left\{0\leq z\leq \bar{R}_{x,y}\left(X^{k-1}(s^-),r \right) \right\}}  \right| \\
		&\qquad+ \|X^k(s^-,r) - X^{k-1}(s^-,r)\| \1_{\left\{0\leq z\leq \bar{R}_{x,y}\left(X^{k-1}(s^-),r \right) \right\}} \\
		&\qquad+ \|y - X^k(s^-,r) \| \left|\1_{\left\{ X^{k}(s^-,r) = x \right\}} - \1_{\left\{ X^{k-1}(s^-,r) = x \right\}}  \right|\1_{\left\{0\leq z\leq \bar{R}_{x,y}\left(X^{k-1}(s^-),r \right) \right\}}.
	\end{align*}
	Next insert this estimate into the right-hand side of \eqref{eq.inequality_1} and use Lemma \ref{F_t_martingale} to further simplify. For example, applying Lemma \ref{F_t_martingale} to the first term yields
	\begin{multline*}
	\mathcal{E}\left[\mathbb{E}\left[\int_0^T \int_0^\infty\|y-X^k(s^-,r) \|\left|\1_{\left\{0\leq z\leq \bar{R}_{x,y}\left(X^k(s^-),r \right) \right\}} - \1_{\left\{0\leq z\leq \bar{R}_{x,y}\left(X^{k-1}(s^-),r \right) \right\}}  \right|\mathcal{N}_{x,y}(dz,ds)\right]\right]\\ = \mathcal{E}\left[\mathbb{E}\left[\int_0^T \|y-X^k(s^-,r) \|\left|\bar{R}_{x,y}\left(X^k(s^-),r \right) -\bar{R}_{x,y}\left(X^{k-1}(s^-),r \right) \right|\,ds\right]\right]
	\end{multline*}
	Estimating the remaining terms in the same way thus yields
		\begin{align}\label{eq.first_key_inequality}
		\| X^{k+1}-X^k \|_{\mathcal{M}_T} &\leq  \Denis{\sum_{y\in S^K} \sum_{\substack{x \in S^K, \\ x \neq y} } \left\{\hat{A}_{x,y}(T) + \hat{B}_{x,y}(T) + \hat{C}_{x,y}(T)\right\},}
		\end{align}
		where 
		\begin{equation*}
		\hat{A}_{x,y}(T) := \mathcal{E}\left[ \mathbb{E}\left[ \int_0^T \|  y - X^k(s^-,r) \| \left|\bar{R}_{x,y}\left(X^k(s^-),\,r \right) - \bar{R}_{x,y}\left(X^{k-1}(s^-),\,r \right) \right| ds \right] \right],
		\end{equation*}
		\begin{equation*}
		\hat{B}_{x,y}(T) := \mathcal{E}\left[ \mathbb{E}\left[ \int_0^T  \lvert \lvert  X^{k}(s^-,r) - X^{k-1}(s^-,r) \rvert \rvert\, \left|\bar{R}_{x,y}\left(X^{k-1}(s^-),\,r \right) \right| ds \right] \right],
		\end{equation*}
		and 
		\begin{equation*}
		\hat{C}_{x,y}(T) := \mathcal{E}\left[ \mathbb{E}\left[ \int_0^T \left|\1_{\left\{ X^{k}(s^-,r) = x \right\}} - \1_{\left\{ X^{k-1}(s^-,r) = x \right\}}  \right| \left|\bar{R}_{x,y}\left(X^{k-1}(s^-),r \right) \right|{ds} \right] \right].
		\end{equation*}
		Since $\lvert \lvert  y - X^k(s^-,r) \rvert \rvert \leq 1$, 
		\begin{equation}\label{eq.estA}
		\hat{A}_{x,y}(T) \leq \mathcal{E}\left[ \mathbb{E}\left[ \int_0^T  \left|\bar{R}_{x,y}\left(X^k(s^-),\,r \right) - \bar{R}_{x,y}\left(X^{k-1}(s^-),\,r \right) \right| {ds} \right] \right].
			\end{equation}
		Estimating the integrand in \eqref{eq.estA} using the Lipschitz continuity of $\Phi_{x,y}$ and the boundedness of $W_{x,y}$ yields
		\begin{multline*}
		\left|\bar{R}_{x,y}\left(X^k(s^-),\,r \right) - \bar{R}_{x,y}\left(X^{k-1}(s^-),\,r \right) \right| \leq \\ L \| W_{x,y} \|_{\infty} \int_{\Gamma} \left|\mathbb{P}\left[X^k(s^-,r') = \psi(x,y) \right] - \mathbb{P}\left[X^{k-1}(s^-,r') = \psi(x,y) \right]\right|dq(r').
		\end{multline*}
		For each fixed $r'$ and each fixed $s \geq 0$, we have 
		\begin{align*}
		\left|\mathbb{P}\left[X^k(s^-,r') = \psi(x,y) \right] - \mathbb{P}\left[X^{k-1}(s^-,r') = \psi(x,y) \right]\right| & \\ &\hspace{-90pt} \leq\mathbb{E}\left[ |\1_{ \{X^k(s^-,r') = \psi(x,y)\} } - \1_{\{X^{k-1}(s^-,r') = \psi(x,y)\}}| \right] \\ &\hspace{-90pt} \leq \mathbb{E}\left[ \| X^{k}(s^-,r') - X^{k-1}(s^-,r') \| \right],
		\end{align*}
		since \[
		\| X^{k}(s^-,r') - X^{k-1}(s^-,r') \| = \begin{cases}
		0, \quad &X^{k}(s^-,r') = X^{k-1}(s^-,r'), \\
		1, &X^{k}(s^-,r') \neq X^{k-1}(s^-,r').
		\end{cases}
		\]
		Thus 
		\begin{align*}
		\left|\bar{R}_{x,y}\left(X^k(s^-),\,r \right) - \bar{R}_{x,y}\left(X^{k-1}(s^-),\,r \right) \right| & \\ &\hspace{-100pt}\leq 
		L \| W_{x,y} \|_{\infty} \int_{\Gamma} \mathbb{E}\left[ \| X^{k}(s^-,r') - X^{k-1}(s^-,r') \| \right]\,dq(r').
		\end{align*}
		Returning to equation \eqref{eq.estA} and using the estimate above gives 
		\begin{align*}
		\hat{A}_{x,y}(T) &\leq L\, \| W_{x,y}\rvert \|_{\infty}\, \int_0^T \int_{\Gamma} \mathbb{E}\left[ \| X^{k}(s^-,r') - X^{k-1}(s^-,r') \| \right]\,dq(r')\, {ds}\\
		&= L\, \| W_{x,y}\|_{\infty}\, \int_0^T \| X^k - X^{k-1} \|_{\mathcal{M}_t}dt.
		\end{align*}
		The estimation for $\hat{B}_{x,y}$ follows simply from using the boundedness of $\Phi_{x,y}$ to obtain
		\begin{align*}
		\hat{B}_{x,y}(T) &\leq ||\Phi_{x,y}||_{\infty}\, \int_0^T \mathcal{E}\left[ \mathbb{E}\left[ \|  X^{k}(s^-,r) - X^{k-1}(s^-,r) \| \right] \right]  {ds} \\
		&= \|\Phi_{x,y}\|_{\infty}\, \int_0^T \| X^k - X^{k-1} \|_{\mathcal{M}_t}dt.
		\end{align*}
		Similarly, $\hat{C}_{x,y}$ can be bounded as follows:
		\begin{align*}
		\hat{C}_{x,y}(T) &\leq \mathcal{E}\left[ \mathbb{E}\left[ \int_0^T \left|\1_{\left\{ X^{k}(s^-,r) = x \right\}} - \1_{\left\{ X^{k-1}(s^-,r) = x \right\}}  \right| \left|\bar{R}_{x,y}\left(X^{k-1}(s^-),r \right) \right| {ds} \right] \right] \\ & \leq \|\Phi_{x,y}\|_\infty \int_0^T \mathcal{E}\left[ \mathbb{E}\left[ \left|\1_{\left\{ X^{k}(s^-,r) = x \right\}} - \1_{\left\{ X^{k-1}(s^-,r) = x \right\}}  \right|  \right] \right]{ds} \\
		&\leq \|\Phi_{x,y}\|_\infty \int_0^T \mathcal{E}\left[\mathbb{E}\left[ \| X^{k}(s^-,r)  -  X^{k-1}(s^-,r) \|  \right] \right]{ds}\\ 
		&= \|\Phi_{x,y}\|_\infty \int_0^T \| X^{k} -  X^{k-1} \|_{\mathcal{M}_t}  dt.
		\end{align*}
		Combining the estimates for $\hat{A}_{x,y}$, $\hat{B}_{x,y}$ and $\hat{C}_{x,y}$ yields
		\[
		\| X^{k+1}-X^k \|_{\mathcal{M}_T} \leq \bar{K}\, \int_0^T \| X^{k}-X^{k-1} \|_{\mathcal{M}_t}dt,\quad \mbox{for each }k\geq 1,
		\]
		where 
		\[
		\bar{K} = \sum_{y\in S^K} \sum_{\substack{x \in S^K, \\ x \neq y} }\left( L\, \|W_{x,y}\|_\infty + 2\|\Phi_{x,y}\|_\infty \right).
		\]
		Therefore 
		\[
		\| X^{k+1}-X^k \|_{\mathcal{M}_T} \leq \frac{(\bar{K}T)^k}{k!} \| X^{1}-X^0 \|_{\mathcal{M}_T} \leq \frac{(\bar{K}T)^k}{k!}, \quad k \geq 1.
		\]
		Thus $\left(X^k\right)_{k \geq 0}$ is a Cauchy sequence in the complete space $\mathcal{M}_T$, and converges to a limit in $\mathcal{M}_T$. The conclusion of the theorem is standard at this point (see, e.g., \cite{revuz2013continuous}) and yields the existence of an $\mathcal{F}_t$-adapted $(\Omega\times\Gamma)$-measurable process $\bar{X} = \left\{ \bar{X}(t,r):\,r\in\Gamma,\, t \in [0,T]  \right\} \in \mathcal{M}_T$ such that $\bar{X} = \Psi\left(\bar{X}\right)$. Therefore $\bar{X}$ is a strong solution to \eqref{eq.mean_field_existence} on $[0,T]$.
		
		The estimates above can be used to show that for any two solutions $\bar{X}$ and $\tilde{X}$ to \eqref{eq.mean_field_existence}, 
		\[ \|\bar{X}-\tilde{X}\|_{\mathcal{M}_T}  \leq K'  \int_0^T \|\bar{X}-\tilde{X}\|_{\mathcal{M}_s}ds.  \]
		Using Gronwall's lemma and the fact that the two solutions have identical initial conditions allow us to conclude that $\bar{X}=\tilde{X}$ in the norm $\|\cdot\|_{\mathcal{M}_T}$ and hence solutions to \eqref{eq.mean_field_existence} are unique in the sense of Definition \ref{def_unique}.
	\end{proof}
	\subsection{Convergence towards the McKean-Vlasov equation}\label{proof:Convergence}
	We now undertake the proof of the convergence result in Theorem \ref{thm_convergence_Kstates}. 
	
	\begin{proof}[Proof of Theorem~\ref{thm_convergence_Kstates} I. Convergence]
		Fix $N\in \N$, $i\in \{1,\dots,N\}$, a time $\tau>0$, and a fixed configuration of sites $(r_1,\dots,r_N)$ drawn as independent and identically distributed random variables according to the probability measure $q$ on $\Gamma$. We proceed to demonstrate that the process $X$ which solves \eqref{eq.network_4_state} converges almost surely with respect to $\mathbb{P}\otimes q$ to a \emph{coupled} process $\bar{X}^i$ a particular solution of equation~\eqref{eq.mean_field_existence} with the same initial condition and Poisson processes as site $i$:
		\begin{multline*}
		\bar{X}^i(t,r) = \xi_0^i(r) + \\ \sum_{y\in S^K} \int_0^t (y-\bar{X}^i(s^{-},r)) \sum_{\substack{x \in S^K, \\ x \neq y} }\1_{ \{ \bar{X}^i (s^-,r) = x\} } \int_0^\infty \1_{\left\{0\leq z\leq \bar{R}_{x,y}\left(\bar{X}^i(s^-),\,r \right) \right\}} \mathcal{N}^i_{x,y}(dz,ds),
		\end{multline*}
		
		Estimate the distance between $X^i$ and $\bar{X}^i$ in the norm $\|\cdot\|_{\mathcal{M}_\tau}$. Using the same arguments as those that established the inequalities \eqref{eq.inequality_1} and \eqref{eq.first_key_inequality} above, and noting that the identical initial conditions simply cancel, we obtain:
		\begin{align}\label{eq.initial_est_spatial}
		\mathbb{E}\left[ \sup_{t \in [0,\tau]} \| X^{i}(t) - \bar{X}^i(t,r_i) \|  \right] &\leq \sum_{y \in S^K} \sum_{\substack{x \in S^K \\ x \neq y}} A_{x,y}(\tau) + B_{x,y}(\tau) + C_{x,y}(\tau),
		\end{align}
		where
		\begin{equation*}
		A_{x,y}(\tau) = \mathbb{E}\left[\int_0^\tau  \| y - X^i(s) \| \,\left| \1_{\left\{X^i(s^-)= x\right\}} - \1_{\left\{\bar{X}^i(s^-,r_i)= x\right\}}  \right|\, \left|R_{x,y}^{i,N}\left(X(s^-)\right) \right|ds \right],
		\end{equation*}		
		\begin{equation*}
		B_{x,y}(\tau) = \mathbb{E}\left[\int_0^\tau \| y - X^i(s^-) \| \, \left|\bar{R}_{x,y}\left(\bar{X}^i(s^-),r_i\right)
		-R_{x,y}^{i,N}\left(X(s^-)\right) \right| ds \right],
		\end{equation*}	and
		\begin{equation*}
		C_{x,y}(\tau) = \mathbb{E}\left[\int_0^\tau  \| \bar{X}^i(s^-,r_i) - X^i(s^-) \| \,\left| \1_{\left\{\bar{X}^i(s^-,r_i)= x\right\}}  \right|\, \left|\bar{R}_{x,y}\left(\bar{X}^i(s^-),r_i\right) \right| ds \right].
		\end{equation*}	
		We immediately obtain an upper bound on $A_{x,y}$ since $\| y - X^i(s) \| \leq 1$ and $R_{x,y}^{i,N}$ is bounded due to the boundedness of each $\Phi_{x,y}$. In particular,
		\begin{align*}
		\mathcal{E}\left[A_{x,y}(\tau)\right] &\leq \|\Phi_{x,y}\|_\infty \, \mathcal{E}\left[\mathbb{E}\left[ \int_0^\tau  \left| \1_{\left\{X^i(s^-)= x\right\}} - \1_{\left\{\bar{X}^i(s^-,r_i)= x\right\}}  \right|\,ds \right]\right] \\
		&\leq \|\Phi_{x,y}\|_\infty  \int_0^\tau \mathcal{E}\left[\mathbb{E}\left[\sup_{u \in [0,s]}  \| X^i(u^-) - \bar{X}^i(u^-,r_i) \|\right]\right]\,ds.
		\end{align*}
		Similarly, it is straightforward to derive the following estimate on $C_{x,y}$:
		\[
		\mathcal{E}\left[C_{x,y}(\tau)\right] \leq \|\Phi_{x,y}\|_{\infty} \int_0^\tau \mathcal{E}\left[\mathbb{E}\left[ \sup_{u \in [0,s]}\| \bar{X}^i(u^-,r_i) - X^i(u) \| \right]\right] ds.
		\]
		The requisite estimation for $B_{x,y}$ is nontrivial; we claim that 
		\begin{equation}\label{eq.key_est_spatial}
		\mathcal{E}\left[ B_{x,y}(\tau) \right] \leq \\ 
		L\, \|W_{x,y}\|_\infty\, \int_0^\tau \mathcal{E}\left[\mathbb{E}\left[ \sup_{u \in [0,s]} \| \bar{X}^i(u,r_i)-X^i(u) \| \right] \right]ds + \tau\, L\,\sqrt{\frac{C}{N}},
		\end{equation}
		for some constant $C>0$ which is independent of $X^i$, $\bar{X}^i$ and $\mathcal{A}_N$. To establish that \eqref{eq.key_est_spatial} holds, begin by estimating as follows: 
		\begin{align}\label{eq.exp.split}
		\mathcal{E}\left[B_{x,y}(\tau)\right] &\leq \int_0^\tau \mathcal{E}\left[\mathbb{E}\left[ \left|\bar{R}_{x,y}\left(\bar{X}^i(s^-),r_i\right)
		-R_{x,y}^{i,N}\left(X(s^-)\right) \right|\right]\right] \,ds.
		\end{align}
		Next make the following estimate on the difference of the transition rate functions using the Lipschitz continuity of $\Phi_{x,y}$:
		\begin{align*}
		\left|\bar{R}_{x,y}\left(\bar{X}^i(s^-),r_i\right)
		-R_{x,y}^{i,N}\left(X(s^-)\right) \right| &\leq \\ &\hspace{-185pt} L \left| \frac{1}{N}\sum_{j=1}^N W_{x,y}(r_i,r_j) \1_{\{ X^j(s^-) = \psi(x,y) \}} - \int_\Gamma W_{x,y}(r_i,r') \mathbb{P}\left[ \bar{X}^i(s^-,r') = \psi(x,y) \right]dq(r') \right|.
		\end{align*}
		Thus 
		\begin{align*}
		\left|\bar{R}_{x,y}\left(\bar{X}^i(s^-),r_i\right)
		-R_{x,y}^{i,N}\left(X(s^-)\right) \right| &\leq \\ &\hspace{-120pt}\frac{L}{N}\sum_{j=1}^N W_{x,y}(r_i,r_j) \left|\1_{\{ X^j(s^-) = \Psi(x,y) \}} - \1_{\{ \bar{X}^i(s^-,r_j) = \psi(x,y)\}}  \right| + \\
		&\hspace{-185pt} L \left| \frac{1}{N}\sum_{j=1}^N W_{x,y}(r_i,r_j) \1_{\{ \bar{X}^i(s^-,r_j) = \psi(x,y) \}} - \int_\Gamma W_{x,y}(r_i,r') \mathbb{P}\left[ \bar{X}^i(s^-,r') = \psi(x,y) \right]dq(r') \right| \\
		&=: D^1_{x,y}(s) + D^2_{x,y}(s).
		\end{align*}
		The first term above is easily estimated using the boundedness of the kernels $W_{x,y}$:
		\begin{align*}
		\mathcal{E}\left[ \mathbb{E}\left[ D^1_{x,y}(s) \right]\right] &\leq  \frac{L\,\|W_{x,y}\|_{\infty}}{N}\sum_{j=1}^N \mathcal{E}\left[\mathbb{E}\left[\left|\1_{\{ X^j(s^-) = \psi(x,y) \}} - \1_{\{ \bar{X}^i(s^-,r_j) = \psi(x,y)\}}  \right|\right]\right] \\ &\leq \frac{L\,\|W_{x,y}\|_{\infty}}{N}\sum_{j=1}^N \mathcal{E}\left[\mathbb{E}\left[\| X^j(s^-) - \bar{X}^i(s^-,r_j) \|\right]\right].
		\end{align*}
		For each $j\in \{1,\dots,N \}$ and each fixed $s\geq 0$, the random variables\\ $\| X^j(s^-)-\bar{X}^i(s^-,r_j) \|$ are identically distributed given the configuration $\mathcal{A}_N$. In other words, the quantity
		\[
		\mathcal{E}\left[ \mathbb{E}\left[  \|X^j(s^-)-\bar{X}^i(s^-,r_j) \| \right] \right]
		\]
		does not depend on $j$. Hence
		\begin{equation*}
		\mathcal{E}\left[ \mathbb{E}\left[ D^1_{x,y}(s) \right]\right] \leq  L\,\|W_{x,y}\|_{\infty}\,\,\mathcal{E}\left[\mathbb{E}\left[\| X^k(s^-) - \bar{X}^i(s^-,r_k) \|\right]\right],
		\end{equation*}
		for each $k \in \{1,\dots,N\}$. It follows readily that 
		\begin{equation}\label{eq.est_pt1}
		\int_0^\tau \mathcal{E}\left[ \mathbb{E}\left[ D^1_{x,y}(s) \right]\right]\,{ds} \leq  L\,\|W_{x,y}\|_{\infty}\int_0^\tau\mathcal{E}\left[\mathbb{E}\left[ \sup_{u \in [0,s]}\| X^k(u^-) - \bar{X}^i(u^-,r_k) \| \right]\right]\,ds,
		\end{equation}
		for each $k \in \{1,\dots,N\}$. 
		
		In order to estimate $D^2_{x,y}(s)$, note that the collection of random variables\\ $\left( W_{x,y}(r_i,r_j) \1_{\{\bar{X}^i(s^-,r_j) = \psi(x,y)\}} \right)_{j \in \{1,\dots,N \}}$ are conditionally i.i.d. given $r_i$. Hence the expression
		\begin{equation}\label{eq.sum_iid_conditional}
		\frac{1}{N}\sum_{j=1}^N W_{x,y}(r_i,r_j) \1_{\left\{ \bar{X}^i(s^-,r_j) = \psi(x,y) \right\}} 
		\end{equation}
		is, conditionally on $r_i$, the sum of i.i.d. random variables with finite mean and variance (since the process $\bar{X}^i$ is a.s. bounded). It is then natural to define $\hat{\mathbb{E}}_i$, the conditional expectation operator on $\Omega' \times \Omega$ given $r_i$. The mean of each summand in \eqref{eq.sum_iid_conditional} with respect to $\hat{\mathbb{E}}_i$ is given by
		\begin{equation}\label{eq.mean_conditional}
		\hat{\mathbb{E}}_i\left[ W_{x,y}(r_i,r_j) \1_{\{\bar{X}^i(s^-,r_j) = \psi(x,y)\}} \right] = \int_{\Gamma} W_{x,y}(r_i,r')\mathbb{P}\left[ \bar{X}^i(s^-,r') = \psi(x,y) \right]dq(r').
		\end{equation}
		Let 
		\[
		F^1_{x,y}(s) = \frac{1}{N}\sum_{j =1}^N W_{x,y}(r_i,r_j) \1_{\{\bar{X}^i(s^-,r_j) = \psi(x,y)\}}
		\]
		and 
		\[
		F^2_{x,y}(s) = \int_{\Gamma} W_{x,y}(r_i,r')\mathbb{P}\left[ \bar{X}^i(s,r')=\psi(x,y) \right]dq(r').
		\]
		Thus, for each fixed $s \in [0,\tau]$, H{\"o}lder's inequality yields
		\begin{multline}\label{eq.var.est}
		\hat{\mathbb{E}}_i \left[ \left| F^1_{x,y}(s) - F^2_{x,y}(s) \right| \right]
		\leq \left( \hat{\mathbb{E}}_i \left[ \left( F^1_{x,y}(s) - F^2_{x,y}(s) \right)^2\, \right] \right)^{1/2} \\
		\leq \frac{1}{N} \left( \sum_{j=1}^N \hat{\mathbb{E}}_i \left[ \left(W_{x,y}(r_i,r_j) \1_{\{\bar{X}^i(s^-,r_j) = \psi(x,y)\}} - F^2_{x,y}(s)\right)^2\, \right]  \right)^{1/2},
		\end{multline}
		where we have used that 
		\[
		\Cov\left( W_{x,y}(r_i,r_j) \1_{\{\bar{X}^i(s^-,r_j) = \psi(x,y)\}} ,\,(W_{x,y}(r_i,r_k) \1_{\{\bar{X}^i(s^-,r_k) = \psi(x,y)\}}  \right) = 0,
		\]
		for $j \neq k$, conditional on $r_i$. By \eqref{eq.mean_conditional}, the summands in \eqref{eq.var.est} are variances of (conditionally) i.i.d. random variables. Thus
		\begin{equation*}
		\hat{\mathbb{E}}_i \left[ \left| F^1_{x,y}(s) - F^2_{x,y}(s) \right| \right]
		\leq \frac{1}{\sqrt{N}} \left( \hat{\mathbb{E}}_i \left[ \left(W_{x,y}(r_i,r_k) \1_{\{\bar{X}^i(s^-,r_k) = \psi(x,y)\}} - F^2_{x,y}(s)\right)^2\, \right]  \right)^{1/2}
		\end{equation*}
		for any $k \neq i$. Since $\bar{X}^i$ is a bounded process, we can uniformly bound its variance by a deterministic constant $C>0$. Thus
		\[
		\hat{\mathbb{E}}_i \left[ \left| F^1_{x,y}(s) - F^2_{x,y}(s) \right| \right]  \leq \sqrt{\frac{C}{N}}\quad \mbox{for some }C>0.
		\]
		Therefore, by law of total expectation, 
		\[
		\mathcal{E}\left[\mathbb{E} \left[  \left| F^1_{x,y}(s) - F^2_{x,y}(s) \right| \right]\right] \leq \sqrt{\frac{C}{N}}.
		\]
		Finally, integrate over $[0,\tau]$ to obtain the desired estimate for $D^2_{x,y}$, i.e. 
		\begin{align*}
		\int_0^\tau \mathcal{E}\left[ \mathbb{E} \left[ \left| D^2_{x,y}(s) \right| \right] \right]{ds} &\leq L\,\int_0^\tau \mathcal{E}\left[ \mathbb{E} \left[ \left| F^1_{x,y}(s)-F^2_{x,y}(s) \right| \right] \right]{ds} \leq \tau\,L\,\sqrt{\frac{C}{N}}.
		\end{align*}
		This estimate establishes the claimed inequality for $B_{x,y}$ from equation \eqref{eq.key_est_spatial}. Finally return to \eqref{eq.initial_est_spatial} and apply the estimates on $A_{x,y}$, $B_{x,y}$ and $C_{x,y}$ derived above to conclude that 
		\begin{align*}
		\mathcal{E}\left[ \mathbb{E}\left[ \sup_{t \in [0,\tau]}\| X^i(t)  - \bar{X}^i(t,r_i) \| \right] \right] &\leq K_1 \int_0^\tau \mathcal{E}\left[ \mathbb{E}\left[ \sup_{u \in [0,s]}\| X^i(u)-\bar{X}^i(u,r_i) \| \right] \right]\,ds  + \frac{K_2}{\sqrt{N}},
		\end{align*}
		where 
		\[
		K_1 = \sum_{y \in S^K} \sum_{\substack{x \in S^K \\ x \neq y}} 2\|\Phi_{x,y}\|_\infty + L\,\|W_{x,y}\|_\infty, \quad K_2 = \tau\sqrt{C} L K (K-1)>0.
		\] 
		Finally apply Gronwall's inequality to show that 
		\[
		\mathcal{E}\left[ \mathbb{E}\left[ \sup_{t \in [0,\tau]} \| X^i(t)  - \bar{X}^i(t,r_i) \| \right] \right]dt \leq \frac{K_1 K_2 \tau}{\sqrt{N}}e^{K_1 \tau} + \frac{K_2}{\sqrt{N}},
		\]
		and letting $N\to \infty$ shows that $X^i(t)$ converges $\mathbb{P}$-$q$ a.s. to $\bar{X}^i(t,r_i)$ for each $t \in [0,\tau]$, which implies the convergence in law stated in the theorem. 
	\end{proof}
	
	\subsection{Generalized Kolmogorov equations}
	We now prove Theorem~\ref{thm_convergence_Kstates} II, the characterization of the law of the mean-field limit through the so-called generalized Kolmogorov equations. Classical Kolmogorov equations are linear differential equations governing the evolution in time of the probability to be at a given state for a continuous-time Markov process. The mean-field limit is non-Markovian, as its transition rates depend on the distribution of the solution at all other locations. In Section~\ref{sec:existUnique} we proved that there exists a unique strong solution to the mean-field process. Therefore, the rates of transition between different states given by~\eqref{eq:Rate2} are well defined, since they only depend on that probability distribution, which is in turn well-defined. Let $\bar{X}(t,r)$ be a solution of the mean-field process, recall that $P_x(t,r)=\P[\bar{X}(t,r)=x]$ and define for each pair of distinct states $x,y \in S^K$ the functions:
	\begin{equation*}
	\Lambda_{x,y}(t)=\Phi_{x,y}\left( \int_\Gamma W_{x,\,y} \left(r,\,r'\right)P_{\psi(x,\,y)}(t^-,r')\, dq(r') \right).
	\end{equation*}	
	These functions are well-defined, non-negative and bounded, and define univocally an auxiliary time inhomogeneous Markov process $Y(t)$ with transition rate from state $x$ to state $y\neq x$ at time $t$ given by $\Lambda_{x,y}(t)$. Standard theory from Markov processes ensures that the law of process $Y(t)$, denoted by $Q$, satisfies the usual Kolmogorov Equations (KE). In detail, the transition probabilities
	\[Q_{x,y}(t)=\P[Y(t)=y \,\vert\, X(0)=x]\] 
	satisfy the differential equations:
	\[\begin{cases}
	\frac{d}{dt}Q_{x,y}(t) = \sum_{z\neq y} Q_{x,z}(t)\Lambda_{z,y}(t) - \Lambda_y(t) Q_{x,y}(t) & \text{Forward KEs}\\
	\frac{d}{dt}Q_{x,y}(t) = \sum_{z\neq x} \Lambda_{x,z}(t) Q_{z,y}(t) - \Lambda_x(t)Q_{x,y}(t) & \text{Backward  KEs}
	\end{cases}\]
	with $\Lambda_x(t):=\sum_{z\neq x}\Lambda_{x,z}(t)$ for each $t \geq 0$.
	
	The probability distribution of $Y(t)$ is equal to that of the non-Markovian McKean-Vlasov process~\eqref{eq.mean_field_existence} because the Markov process $Y$ has (by definition) independent exponentially distributed transitions with rates $\Lambda_{x,y}(t)$, exactly as the McKean-Vlasov process $\bar{X}$. Therefore, the transition probabilities 
	\[P_{x,y}(t,r)=\P[\bar{X}(t,r)=y \,\vert\, \bar{X}(0,r)=x]\]
	are equal to $Q_{x,y}(t)$. The linear Kolmogorov equations for the transition probabilities of $Y$ thus convert into nonlinear equations for those of $\bar{X}$:
	\[\begin{cases}
	\partial_t P_{x,y}(t,r) = \sum_{z\neq y} P_{x,z}(t,r)\Lambda_{z,y}(t) - \Lambda_y(t) P_{x,y}(t,r) & \text{Forward GKEs}\\
	\partial_t P_{x,y}(t,r) = \sum_{z\neq x} \Lambda_{x,z}(t) P_{z,y}(t,r) - \Lambda_x(t)P_{x,y}(t,r) & \text{Backward  GKEs}.
	\end{cases}\]
	Given that $P_x(t,r)=\sum_{y}P_{y,x}(t,r)\mu_0(r,y)$ with $\mu_0(r,y)=\P[\bar{X}(r,0)=y]$, we thus obtain
	\begin{equation}\label{eq.GKE}
	\partial_t P_{x}(t,r) = \sum_{z\neq x} P_{z}(t,r) \Lambda_{z,x}(t)  - \Lambda_x(t)P_{x}(t,r), \quad x \in S^K,
	\end{equation}
	using the forward GKEs above. The system of equations given by \eqref{eq.GKE} is precisely the self-consistent $K$-dimensional nonlinear system of IDEs stated in the theorem. 
	\subsection{Propagation of chaos}
	We complete the proof of Theorem~\ref{thm_convergence_Kstates} by proving the propagation of chaos property. In Section~\ref{proof:Convergence}, we showed an almost sure convergence of $X^i$ towards the coupled mean-field process $\bar{X}^i(\cdot,r_i)$ introduced in the proof. For $i\neq j$, because the initial conditions $\xi^i(r)$ and $\xi^j(r)$ are independent (assumption $H_{IC}$) and for any $(x,x',y,y')$ elements of $S^K$, the Poisson processes $\mathcal{N}^{i}_{x,y}$ and $\mathcal{N}^{j}_{x',y'}$ are independent, the processes $\bar{X}^i(\cdot,r_i)$ and $\bar{X}^j(\cdot,r_j)$ are independent. Similarly, for any sequence of $p$ distinct indices $(i_1,\dots,i_p)$, the processes $\bar{X}^{i_1}(t,r_{i_1}),\dots, \bar{X}^{i_p}(t,r_{i_p})$ are mutually independent, and the collection of processes $(X^{i_1},\dots,X^{i_p})$ converge almost surely towards $\left(\bar{X}^{i_1}(\cdot,r_{i_1}),\dots, \bar{X}^{i_p}(\cdot,r_{i_p})\right)$, completing the proof.
	
	\newpage
	\appendix
	\section{Numerical Methods \& Codes}\label{sec.appendix}

	In this appendix we describe in more detail the theory and numerical methods behind the figures in the main text. All code needed to reproduce the figures is available on \href{https://github.com/Touboul-Lab/spatial_models_Staver_Levin}{github.com/Touboul-Lab} and was run on MATLAB version R2019B. In figures where the bifurcation diagram of an ODE was overlaid on the result of a simulation, the bifurcation diagram was computed in AUTO~\cite{ermentrout2002simulating} and the axes aligned in Inkscape.
	\subsection{Numerical parameters}
	For all numerical calculations and simulations we take the following smooth functional forms for the fire threshold functions $\phi$ and $\omega$:
	\[
	\omega(x) = \omega_0 + \frac{\omega_1-\omega_0}{1 + e^{-(x-\theta_1)/s_1}},\quad \quad \phi(x) = \phi_0 + \frac{\phi_1-\phi_0}{1 + e^{-(x-\theta_2)/s_2}} \quad \mbox{for } x\in[0,1].
	\]
	with parameter values as given in Table \ref{table.sigmoids} below.
	\begin{table}[H]
		\centering
		\begin{tabular}{lcccccccccc}
			Parameter & $\mu$ & $\nu$ & $\omega_0$ & $\omega_1$ & $t_1$ & $s_1$ & $\phi_0$ & $\phi_1$ & $t_2$ & $s_2$ \\ \hline
			Value & 0.1 & 0.05   & 0.9        & 0.4        & 0.4   & 0.01  & 0.1      & 0.9      & 0.4   & 0.05 
		\end{tabular}
		%\caption{Typical parameter values for the fire response functions.}
		\label{table.sigmoids}
	\end{table}
	\subsection{Figure \ref{two_state_expected_bifurcations}}
	
	Figure \ref{two_state_expected_bifurcations} shows the results of simulations of the macroscale interacting particle system model for different values of the seed dispersal intensity $\bar{J}$ across a range of initial conditions. There is a single patch $M=1$ and $N = 3000$ sites in each simulation and we used the Gillespie algorithm to simulate transition times. For each value of $\bar{J}$ shown, we ran $20$ simulations with the initial grass/forest proportions randomly assigned with a weighting which we varied from $0$ to $1$ to observe a full range of initial conditions. Each blue starred point is the proportion of grass in the particle system at time $t=100$. The red curves are the bifurcation diagram in $\bar{J}$ of the ODE, i.e. the GKEs of the mean-field limit:
	\begin{align*}
	\frac{d}{dt}P_G(t) &= \left(1-P_G(t)\right) \phi\left( P_G(t) \right) - \bar{J}\, P_G(t) \left(1 - P_G(t)\right)
	\end{align*}
	with $P_G(t)$ the probability of any given site in the patch to be covered by grass at time $t$. Solid red lines are stable equilibrium curves and dashed red lines denote unstable equilibria; solid red dots mark the bifurcation points.
	
	When multiple distinct states were observed in the particle system at time $t=100$ we recorded the proportions of grass in the initial conditions and plotted with a solid blue dot the lowest proportion whose corresponding simulation ended up on the upper branch of the bifurcation diagram. Similarly, the highest initial proportion which resulted in a time $t=100$ state on the lower branch is recorded with a solid green dot. Tracking which initial conditions result in which end states gives an idea of the ``basin of attraction'' for the transient states of the particle system - we see a close correspondence between this informal ``stochastic attraction basin'' and the basin of attraction of the stable states of the GKEs.
	
	\subsection{Figure \ref{fig.qsd}} Panels A1 and A2 are obtained by directly computing the principal eigenvalue and corresponding normalized eigenvector of the transition intensity matrix $Q$ of the irreducible part of the Markov chain (i.e. the finite-size model).
	
	\subsection{Figure \ref{fig.periodic_IPS} }
	Figure \ref{fig.periodic_IPS} shows a comparison between the macroscale particle system model with $M=1$ patch and $N=3000$ sites, and the corresponding GKEs of the mean-field limit, i.e.
	\begin{align*}
	\dot{G} &= \mu S + \nu T + \phi(G)F - \bar{J}\, G F - \beta\, G T, \\
	\dot{S} &= -\mu S - \omega(G) S - \bar{J}\, S F + \beta\, G T,\\
	\dot{T} &= -\nu T + \omega(G) S - \bar{J}\, T F, \\
	\dot{F} &= \bar{J}\, (G + S + T)F - \phi(G)F, \\
	1 &= G + S + T + F.
	\end{align*}
	In panel A1, we show the solution to the GKEs versus time - the GKEs are a the system of ODEs and are solved using an explicit Euler method with step size $0.01$; parameters are as in Table \ref{table.sigmoids} with $\bar{J} = 0.25$ and $\beta = 0.4$. Panel C1 also shows solutions of the GKEs solved via the Euler method but now $\beta=0.4$ is fixed while $\bar{J}$ is varied to show solutions approaching a heteroclinic cycle in the phase space. 
	
	Panels A2, B and C2 show the results of direct simulations of the four species macroscale particle system using the Gillespie algorithm to simulate transition times. We have $M=1$ patch and $N=3000$ sites with $\bar{J}$ and $\beta$ as in panel $A1$. Panel A2 shows the evolution of the proportions of sites in each state while panel B shows the full solution with the state of each site at each time recorded - the sites are essentially in a synchronized oscillation. Panel C2 shows multiple simulations of the macroscale particle system with the value of $\bar{J}$ varying between simulations and $\beta = 0.4$ fixed.	
	
	\subsection{Figure \ref{fig.maxwell}}
	
	Figure \ref{fig.maxwell} presents simulations of the mesoscale finite-size model and the corresponding GKEs on the periodic spatial domain $\Gamma = \mathcal{S}_5$ (represented by the interval $[0,\,5]$ with the endpoints identified). 
	
	For panels A1, A2 and A3 we directly simulated the finite-size system using the Gillespie algorithm. Panels A4, A5 and A6 are solutions of the GKEs of the mean-field limit. In particular, we solve the nonlinear integro-differential equation
	\begin{multline}\label{eq.KE_appendix}
	\frac{\partial}{\partial t}G(t,r) =  \left(1 - G(t,r)\right) \,\phi\left( \int_{\Gamma} W(r',r)\,G(t,r')\, dr' \right) \\ - G(t,r) \,\left(1 - \int_{\Gamma} J(r',r)\, G(t,r') \,dr' \right), \quad (t, r)\in \mathbb{R}_+ \times \Gamma,
	\end{multline}
	with periodic boundary conditions and kernels as described in Section \ref{sec.waves}. We solve equation \eqref{eq.KE_appendix} by discretizing time with an explicit Euler scheme to obtain
	\begin{equation}
	\begin{split}\label{eq.KE_appendix_euler}
	G(n+1,r) = G(n,r) +  h \left\{ \left(1 - G(n,r)\right) \,\phi\left( \int_{\Gamma} W(r',r)\,G(n,r')\, dr' \right) \right. \\
	\left. - G(n,r) \,\left(1 - \int_{\Gamma} J(r',r)\, G(n,r') \,dr' \right) \vphantom{\left(1 - G(n,r)\right) \,\phi\left( \int_{\Gamma} W(r',r)\,G(n,r')\, dr' \right)} \right\},\quad (n, r)\in \mathbb{Z}^+ \times \Gamma.
	\end{split}
	\end{equation}
	We then discretize the integrals using the 1D trapezoidal rule and approximate the solution on the evenly spaced grid $\{0,\Delta,2\Delta,\dots,5\}$ for some $\Delta>0$. In practice, we found that a step size $h<0.1$ and $200$ spatial grid points was sufficient to ensure numerical stability of the scheme, and the scheme remained stable as we decreased the time step and increased the number of grid points. 
	
	In part B of Figure \ref{fig.maxwell} we compare the results of simulations of the mesoscale particle system and the corresponding generalized Kolmogorov equations with a \emph{nonuniform initial site distribution}. Panel B1 is a space-time plot of the result from a direct simulation of the mesoscale particle system with $N = 2000$ sites whose initial positions are drawn according to the distribution
	\begin{equation}\label{eq.init_dist_appendix}
	dq(x)=\left(a+b\,x\right)\mathbbm{1}_{[0,L]}(x)\,dx,\quad a,b>0, \quad x \in [0,1]
	\end{equation}
	with $a=0.4$ and $b=1.2$ for this particular example. The spatial domain $\Gamma$ is now the interval $[0,1]$, and the dispersal kernels are standard Gaussians and of convolution type, i.e.
	\[
	W(r,r') = \frac{J(r,r')}{\bar{J}} = \frac{1}{\sigma\sqrt{2\pi}}e^{\frac{-(r-r')^2}{2\sigma^2}},\quad r,r' \in \mathbb{R}.
	\]
	We take $\bar{J} = 1.1$ in these simulations. In contrast to panel A, which has periodic boundary conditions, the boundary conditions for this simulation are ``reflecting'' in the following sense: consider a function $u$ defined on $[0,L]$. First extend $u$ from a function on $[0,L]$ to $[-L,L]$ by reflection:
	\[
	u_{R}(x,t) = \begin{cases}
	u(x,t), \quad &x\in[0,L],\\
	u(-x,t), &x\in[-L,0].
	\end{cases}
	\]
	Now let $\tilde{u}$ be the standard $2L$ periodic extension of $u_R$ (as defined above) so that $\tilde{u}$ is defined on all of $\mathbb{R}$. When the system has a heterogeneous spatial structure as in this example, since sites are less densely packed close to zero and more densely packed situated closer to $L$, a periodic extension introduces unrealistic boundary effects. In particular, periodic boundaries put dense regions adjacent to sparse sites regions while the reflecting boundary makes sure that sparse regions neighbor sparse regions at the boundaries. When dispersal kernels are sufficiently local, this avoids unintuitive and physically unrealistic solutions with significant boundary effects. For the simulation of the particle system, we once more use the Gillespie algorithm to simulate transition times.
	
	Panel B2 of Figure \ref{fig.maxwell} is a space-time plot of the solution of the GKEs for the mean-field mesoscale system with initial site distribution given by \eqref{eq.init_dist_appendix}. The integro-differential equations to be solved are given by
	\begin{multline}\label{eq.KE_NU_appendix}
	\frac{\partial}{\partial t}G(t,r) =  \left(1 - G(t,r)\right) \,\phi\left( \int_{\Gamma}(a + b\,r') \,W(r-r')\,G(t,r')\, dr' \right) \\ - G(t,r) \,\left(1 - \int_{\Gamma}(a + b\,r')\, J(r-r')\, G(t,r') \,dr' \right), \quad (t, r)\in \mathbb{R}_+ \times \Gamma,
	\end{multline}
	with reflecting boundary conditions as outlined above. We solved \eqref{eq.KE_NU_appendix} numerically using the same Euler time discretization and trapezoidal rule spatial discretization as before with similar discretization parameters giving good numerical stability of the scheme.
	
	Panel B3 has three components overlaid on the same axes: the time $t=500$ state of the particle system from B1 (solid blue), the time $t=500$ state of the GKE solution from B2 (dashed black line) and the codimension-1 bifurcation diagram in $r$ of the ODE 
	\begin{equation}\label{eq.ODE_limit_appendix}
	\frac{d}{dt}G(t) = \left( 1 - G(t) \right)(a + b\,r)\,G(t) - G(t) \left( 1 - (a + b\,r)\,\bar{J}\,G(t) \right).
	\end{equation}
	The ODE given by \eqref{eq.ODE_limit_appendix} is obtained from the GKE \eqref{eq.KE_NU_appendix} by formally letting $\sigma\to0^+$ in both the fire and dispersal kernels in order to study the case of no spatial interaction between sites. Hence \eqref{eq.ODE_limit_appendix} can be considered as the natural zero dispersal limit of the system with Gaussian kernels. The bifurcation diagram for \eqref{eq.ODE_limit_appendix} is overlaid with solid red lines for stable equilibria and unstable equilibria denoted by dashed red lines.

	\bibliographystyle{siamplain}
	\bibliography{prob_foundations}
	
\end{document}